\documentclass[a4,10pt,twocolumn]{scrartcl}
\pdfoutput=1
\usepackage[top=2.5cm, bottom=2.5cm, left=1.6cm, right=1.6cm]{geometry}

\usepackage[hang]{footmisc}
\setfnsymbol{wiley}

\makeatletter
\renewcommand{\@biblabel}[1]{[#1]\hfill}
\makeatother

\usepackage{fancyhdr}
\fancypagestyle{firststyle}
{
  \fancyhf{}
  \fancyfoot[C]{\footnotesize © 2022 IEEE.  Personal use of this material is permitted.  Permission from IEEE must be obtained for all other uses, in any current or future media, including reprinting/republishing this material for advertising or promotional purposes, creating new collective works, for resale or redistribution to servers or lists, or reuse of any copyrighted component of this work in other works.}
}

\setkomafont{section}{\Large}
\setkomafont{subsection}{\large}
\setkomafont{subsubsection}{\normal}

\newcommand{\IEEEPARstart}[1]{\textbf{\Large #1}}

\usepackage[format=plain,  
labelsep=period,           
font=small,                
labelfont=bf,              
skip=5pt                  
]{caption}
\usepackage{cite}
\usepackage{amsmath,amssymb,amsfonts}
\usepackage{algorithmic}
\usepackage{subcaption}
\usepackage{graphicx}
\usepackage{xfrac}
\usepackage{textcomp}
\usepackage{tikz} 
\usepackage{pgfplots}
\usepackage{xfrac}
\usepackage{mathptmx}
\usepackage{amsthm}
\DeclareMathAlphabet{\mathcal}{OMS}{cmsy}{m}{n}
\usepackage[colorlinks=true,%
linkcolor=black,%
urlcolor=black,%
citecolor=black]{hyperref}
\usepackage{bbm}
\usetikzlibrary{positioning}
\usetikzlibrary{calc}
\usetikzlibrary{decorations.markings,patterns,decorations.pathmorphing}
\usetikzlibrary{arrows, shapes}
\tikzstyle{block} = [draw, thick, node distance=0.5cm, minimum width=1cm, inner sep=6pt]
\tikzstyle{sum} = [draw, thick, circle, node distance=1cm, inner sep=3.5pt, path picture={\node at (path picture bounding box.center) [draw, anchor = center] {$+$};}]

\newtheorem{assumption}{Assumption}
\newtheorem{definition}{Definition}
\newtheorem{theorem}{Theorem}
\newtheorem{proposition}{Proposition}
\newtheorem{lemma}{Lemma}

\newtheorem{remark}{Remark}

\usepackage{authblk}

\usepackage{ulem}

\newcommand{\titlename}{Model predictive control for linear uncertain systems using integral quadratic constraints}

\newcommand{\x}{x}
\renewcommand{\u}{u}
\newcommand{\w}{w}
\newcommand{\y}{y}
\renewcommand{\d}{d}

\newcommand{\constrMat}{H}
\newcommand{\constrVec}{h}
\newcommand{\constrTight}{g}
\newcommand{\dset}{\mathbb{D}}
\newcommand{\dmax}{d_\mathrm{max}}
\newcommand{\A}{A}
\newcommand{\Ae}{A_\G}
\newcommand{\Ce}{C_\G}
\newcommand{\G}{G}
\newcommand{\Bu}{B^\u_\G}
\newcommand{\Bd}{B^\d_\G}
\newcommand{\Bw}{B^\w_\G}
\newcommand{\C}{C}
\newcommand{\Du}{D^\u_\G}
\newcommand{\Dw}{D^\w_\G}
\newcommand{\Dd}{D^\d_\G}
\newcommand{\AG}{A_\G}

\newcommand{\BGd}{B_\G^\d}
\newcommand{\BGw}{B_\G^\w}
\newcommand{\CG}{C_\G}

\newcommand{\DGw}{D_\G^\w}
\newcommand{\DGd}{D_\G^\d}
\newcommand{\nomx}{\xi}
\newcommand{\nomu}{v}
\newcommand{\e}{e}
\newcommand{\K}{K}

\newcommand{\Kloc}{K_{\Omega}}

\newcommand{\Qproof}{\tilde{\Q}}
\renewcommand{\t}{t}
\renewcommand{\k}{k}

\newcommand{\nomy}{r}

\newcommand{\filt}{\Psi}
\newcommand{\filtx}{\psi}
\newcommand{\filtA}{A_\Psi}
\newcommand{\filtBin}{B_{\Psi}^\y}
\newcommand{\filtBout}{B_{\Psi}^\w}
\newcommand{\filtC}{C_\Psi}
\newcommand{\filtDin}{D_{\Psi}^\y}
\newcommand{\filtDout}{D_{\Psi}^\w}
\newcommand{\filty}{p}
\newcommand{\M}{M}
\newcommand{\symLMI}{\star}
\newcommand{\allA}{\mathcal{A}}
\newcommand{\allB}{\mathcal{B}}
\newcommand{\allC}{\mathcal{C}}
\newcommand{\allD}{\mathcal{D}}
\renewcommand{\P}{P}
\newcommand{\Ptube}{\mathcal{P}}

\newcommand{\Q}{Q}
\newcommand{\J}{J}

\newcommand{\Rcost}{R}
\renewcommand{\S}{S}

\newcommand{\T}{T}

\newcommand{\rpred}[2]{\bar{\nomy}_{#1|#2}}
\newcommand{\nomxpred}[2]{\bar{\nomx}_{#1|#2}}
\newcommand{\nomupred}[2]{\bar{\nomu}_{#1|#2}}
\newcommand{\nomxsol}[3][]{{{\nomx}_{#2|#3}^{\star#1}}}
\newcommand{\nomusol}[3][]{{{\nomu}_{#2|#3}^{\star#1}}}
\newcommand{\csol}[3][]{{{\s}_{#2|#3}^{\star#1}}}
\renewcommand{\c}{c}
\newcommand{\s}{s}
\newcommand{\cpred}[2]{\bar{\s}_{#1|#2}}
\newcommand{\dimx}{{n_\x}}
\newcommand{\dimu}{{n_\u}}
\newcommand{\dimw}{{n_\w}}
\newcommand{\dimd}{{n_\d}}
\newcommand{\dimy}{{n_\y}}
\newcommand{\dimfilty}{{n_\filty}}
\newcommand{\dimconstrVec}{{n_\constrVec}}

\newcommand{\ltwo}[1]{\ell_{2}^{#1}}
\newcommand{\ltwoe}[1]{\ell_{2e}^{#1}}
\newcommand{\ltworho}[1]{\ell_{2,\rho}^{#1}}
\newcommand{\dint}{\,\mathrm{d}}
\DeclareMathOperator{\diag}{diag}
\newcommand{\R}{\mathbb{R}}

\newcommand{\RHinf}{\mathbb{RH}_\infty}
\newcommand{\RLinf}{\mathbb{RL}_\infty}
\newcommand{\norm}[1]{\left\|#1\right\|}

\newcommand{\constd}{a_\d}
\newcommand{\constproof}{a}

\newcommand{\ssrep}[4]{\left[\begin{array}{c|c} #1 & #2 \\ \hline #3 & #4\end{array}\right]}

\newcommand{\refeq}[2]{\overset{\makebox[0pt][c]{\scriptsize #1}}{#2}}

\begin{document}\normalem
\title{\titlename}
\author{Lukas Schwenkel$^\text{1}$, Johannes Köhler$^\text{1,2}$, Matthias A. Müller$^\text{3}$, and Frank Allgöwer$^\text{1}$
\thanks{ 
    F. Allgöwer and M. A. Müller are thankful that this work was funded by Deutsche Forschungsgemeinschaft (DFG, German Research Foundation) – AL 316/12-2 and MU 3929/1-2 - 279734922. F. Allgöwer is thankful that this work was funded by Deutsche Forschungsgemeinschaft (DFG, German Research Foundation) – GRK 2198/1 - 277536708. L. Schwenkel thanks the International Max Planck Research School for Intelligent Systems (IMPRS-IS) for supporting him.
    
    $^1$ L. Schwenkel, J. Köhler, and F. Allgöwer are with the Institute for Systems Theory and Automatic Control, University of Stuttgart, Stuttgart 70550, Germany (e-mail: $\{$schwenkel, allgower$\}$@ist.uni-stuttgart.de).
    
    $^2$ J. Köhler is with the Institute for Dynamical Systems and Control, ETH Zurich, ZH-8092, Switzerland (e-mail: jkoehle@ethz.ch).
    
    $^3$ M. A. Müller is with the Institute of Automatic Control, Leibniz University Hannover, 30167 Hannover, Germany (email: mueller@irt.uni-hannover.de).}}

\date{\vspace{-1cm}}

\maketitle
\thispagestyle{firststyle}

\begin{abstract}
\textbf{Abstract: }In this work, we propose a tube-based MPC scheme for state- and input-constrained linear systems subject to dynamic uncertainties characterized by dynamic integral quadratic constraints (IQCs).
In particular, we extend the framework of $\rho$-hard IQCs for exponential stability analysis to external inputs.
This result yields that the error between the true uncertain system and the nominal prediction model is bounded by an exponentially stable scalar system.
In the proposed tube-based MPC scheme, the state of this error bounding system is predicted along with the nominal model and used as a scaling parameter for the tube size.
We prove that this method achieves robust constraint satisfaction and input-to-state stability despite dynamic uncertainties and additive bounded disturbances.
A numerical example demonstrates the reduced conservatism of this IQC approach compared to state-of-the-art robust MPC approaches for dynamic uncertainties.

\end{abstract}

\section{Introduction}

\begin{figure}[t]
  \centering
  \begin{tikzpicture}
    \node [block] (Delta) {\Large $\Delta$};
    \node [block, below=of Delta, minimum height=0.8cm] (G) {\Large $\G$};
    \node[sum, left=of Delta] (sum) {};
    \draw[-latex, thick] (G.east) +(0.8,-0.2) node [right] {$\d$} -- ($(G.east)+(0,-0.2)$);
    \draw[-latex, thick] (Delta.east) -| node[left, pos=0.75] {$\w$} ($(G.east)+(1,0.2)$) -- ($(G.east)+(0,0.2)$);
    \draw[-latex, thick] (sum) -- node[above] {$\y$} (Delta);
    \draw[latex-, thick] (sum) |- (G);
    \draw[-latex, thick] (sum) --  (Delta);
    \draw[-latex, thick] (sum) +(-0.8,0) node [left] {$\nomy$} -- (sum);
  \end{tikzpicture}
  \caption{Feedback interconnection of a linear system $\G$ and an uncertainty $\Delta$ with external inputs $\d$ and $\nomy$, which corresponds to the error dynamics~\eqref{eq:error_interconnection} in the proposed MPC scheme that is defined in Sec.~\ref{sec:setup}.}\label{fig:feedback_interconnection}
\end{figure}
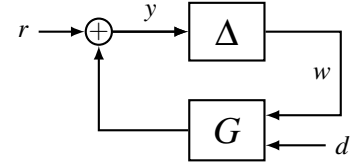

\IEEEPARstart{W}HEN facing a control problem with hard state or input constraints, a popular approach that can guarantee stability and constraint satisfaction is to design a model predictive controller (MPC) (see e.g.~\cite{Rawlings2017} or~\cite{Mayne2014}).
Throughout the past decades, the question how to adjust an MPC scheme to maintain these guarantees in the presence of disturbances or model uncertainties has been studied frequently for different kinds of uncertainties~\cite{Bemporad1999}. 
This led to several robust MPC schemes reaching from bounded disturbances (e.g.~\cite{Chisci2001}) over stochastic disturbances (e.g.~\cite{Mesbah2016}), state and input dependent disturbances (e.g.~\cite{Koehler2021}), and parametric uncertainties (e.g.~\cite{Kouvaritakis2016}) to dynamic uncertainties (e.g.~\cite{Falugi2014}).
The reason for this wealth of approaches is not only the different nature of various uncertainties but also that there is a trade-off between conservatism and complexity of the underlying uncertainty descriptions.
While some control tasks require a fast and simple MPC scheme, there are other scenarios where a larger online computational complexity can be tolerated to gain tighter uncertainty descriptions and less cautious controllers, which can lead to significant performance improvements and much larger operating ranges.
Interestingly, there is a lot of MPC literature on the rather simple case of additive bounded disturbances, whereas, on the other end of the scale, the handling of unmodeled dynamics, delays, errors from using reduced order models, or other dynamic uncertainties in MPC remains an open research area~\cite{Mayne2014}.
This is in contrast to classical robust control literature (e.g.,~\cite{Zhou1995}) where stability and performance of feedback interconnections of a known linear system $\G$ and a dynamic uncertainty $\Delta$ as shown in Fig.~\ref{fig:feedback_interconnection} are studied comprehensively.
We make use of analysis tools from the robust control literature for such interconnections and base our proposed MPC scheme on the powerful and efficient framework of integral quadratic constraints (IQCs, see~\cite{Megretski1997} for the original paper, or~\cite{Veenman2016} for a tutorial overview).
When the input-output behavior of an uncertainty is described by an IQC, stability and performance of the feedback interconnection can be verified with linear matrix inequalities (LMIs).
In this article, we bridge this gap between classical robust control methods and robust MPC by providing an MPC design method for linear constrained systems subject to dynamic uncertainties characterized using $\rho$-hard IQCs as defined in~\cite{Lessard2016}.
Furthermore, the use of IQCs in robust MPC is a general and unifying approach since a multitude of uncertainties can be described with IQCs such as $\ell_2$-gain bounds, uncertain time-delays, polytopic parameter uncertainties, or sector- and slope-restricted nonlinearities.

\paragraph*{Related work}
A widespread approach to robustify MPC schemes is tube-based MPC, where a nominal MPC scheme is implemented with tighter constraints and the amount of the constraint tightening is determined from the size of a tube confining all possible trajectories of the true system.
The main advantage of tube-based MPC compared to other robust MPC approaches like min-max MPC or multi-stage MPC is that the online computational complexity of tube-based MPC schemes is, if at all, only slightly larger than a nominal MPC. 
Tube-based MPC was first introduced by~\cite{Chisci2001} and~\cite{Mayne2001} for linear systems subject to additive bounded disturbances and later improved in~\cite{Mayne2005}.
Instead of using a static tube,~\cite{Rakovic2012} proposed to scale the tube size with a parameter that is optimized online, thereby offering more flexibility.
This idea is also used in~\cite{Fleming2015} to develop an MPC scheme for systems subject to parametric uncertainties, which is in~\cite{Kouvaritakis2016} extended to a mix of parametric uncertainties and bounded additive disturbances.
Recently, in~\cite{Subramanian2021} the tube-based approach to parametric uncertainties is combined with a less conservative multi-stage MPC allowing the user to trade off between complexity and conservatism of the MPC scheme.
In reality, however, uncertainties might not be parametric but are often more complex and dynamic.
In~\cite{Lovaas2008} and ~\cite{Falugi2014}, dynamic uncertainties are captured with a finite $\ell_\infty$-gain and conservatively overapproximated using a constant additive bound in order to use the MPC schemes designed for additive bounded disturbance.
In~\cite{Loehning2014}, a dynamic bound in form of a stable scalar system is used to ensure robust constraint satisfaction and stability when applying MPC with a reduced order model, despite the dynamic uncertainty arising from the model order reduction.
Similarly, in \cite{Thangavel2018} and \cite{Thangavel2019} such an error bounding system is used to describe the dynamic uncertainty and a multi-stage MPC is employed.
However, no guarantees regarding robust constraint satisfaction or stability are provided and the computational demand increases exponentially compared to a nominal MPC. 
Instead of designing a new MPC scheme, existing MPC schemes have been tested in~\cite{Heath2006},~\cite{Petsagkourakis2020a}, and~\cite{Petsagkourakis2020b} for robust stability against dynamic uncertainties satisfying an IQC, however, without guarantees for robust state constraint satisfaction.
Summing up, there is a need for a robust MPC scheme that can guarantee stability and constraint satisfaction for a general class of dynamic uncertainties.
IQCs offer this generality and can describe a wide variety of uncertainty classes.
In this article, we design a robust MPC scheme for systems subject to dynamic uncertainties that are bounded by IQCs and to the best knowledge of the authors, there exist no such MPC schemes ensuring robust stability and constraint satisfaction.

\paragraph*{Contribution and Outline}
We propose a tube-based MPC scheme for state and input constrained linear systems subject to dynamic uncertainties that are described by $\rho$-hard IQCs.
In Sec.~\ref{sec:setup}, we start by describing the problem setup, providing a brief introduction into tube-based MPC and $\rho$-hard IQCs, as well as connecting the time-domain $\rho$-hard IQCs to frequency domain $\rho$-IQCs via Positive-Negative multipliers.
In Sec.~\ref{sec:error_sys}, we extend the framework of $\rho$-hard IQCs to interconnections with external inputs by using a scalar exponentially stable system to bound the state of the extended system.
In our tube-based MPC setup, we show that this scalar system provides an upper bound on the error between the true uncertain system and a nominal prediction model.
In Sec.~\ref{sec:mpc}, we develop a tube-based MPC scheme that predicts the state of this dynamic error bound along with the nominal model and utilizes it as a scaling parameter for the tube size.
This results in a tube dynamics which adjusts its size online according to the excitation of the dynamic uncertainty.
As our key contribution, we prove that the proposed MPC scheme guarantees input-to-state stability (ISS) against bounded external disturbances as well as robust constraint satisfaction despite the dynamic uncertainty in the feedback loop.
Further, in Sec.~\ref{sec:example}, we demonstrate the flexibility and the reduced conservatism of the IQC approach in a numerical example and discuss some implementation aspects.

Preliminary results regarding the incorporation of IQCs in MPC can be found in the conference proceedings \cite{Schwenkel2020}.
Compared to~\cite{Schwenkel2020}, the present article provides a more comprehensive analysis including connections to frequency domain IQCs, a more elaborate example, and a less conservative controller resulting from an improved scheme and a better proof technique.
In particular, the initial MPC design in~\cite{Schwenkel2020} considers a fixed nominal system, and hence the set of nominally feasible control actions is independent of the measured state, thus reducing to a robust trajectory planning with a linear stabilizing feedback.
As one of the main technical contributions of the present work, we extend the tube dynamics to allow for an optimization of the initial state of the nominal system, thus, significantly increasing the flexibility of the proposed approach.

\paragraph*{Notation} We denote the unit circle in the complex plane by $\mathbb{T} = \{z\in\mathbb{C}|\,|z|=1\}$, the set of real rational and proper transfer matrices of dimension $n\times m$ with $\RLinf^{n\times m}$ and its subset of functions analytic outside the closed unit disk with $\RHinf^{m\times n}$. 
Whenever the dimensions are obvious from the context, we write $\RLinf$ and $\RHinf$.
The set of sequences in $\R^n$ is denoted by $\ltwoe{n}=\{(q_k)_{k\in\mathbb{N}}|q_k\in\R^n\}$, the subset of square summable sequences is denoted by $\ltwo{n}=\{q\in\ltwoe{n}| \sum_{k=0}^\infty \norm{q_k}^2 <\infty \}$, and for $\rho\in(0,1)$ the subspace of exponentially square summable sequences is denoted by $\ltworho{n}=\{q\in\ltworho{n}|\sum_{k=0}^\infty \rho^{-2k}\norm{q_k}^2 <\infty \}$.
The $z$-transformation of a sequence $q \in \ltwo{n}$ is denoted by $\hat{q}(z)=\sum_{k=0}^{\infty} q_kz^{-k}$.
For symmetric forms $X^\top PX$ with $P\in\R^{n\times n}$ and $X\in\R^{n \times m}$, we write $[\symLMI]^\top P X$ for convenience. 
For matrices $A,B,C,D$ with suitable dimensions we define $\ssrep{A}{B}{C}{D}=D+C(zI-A)^{-1} B$. 
For $\rho\in(0,1]$ and $\Pi \in \RLinf^{n\times m}$ we define the notation $\Pi_\rho$ as the multiplier $\Pi_\rho :\mathbb C\to \mathbb C^{n\times m}, z \mapsto \Pi(\rho z)$ and further, we denote the para-Hermitian conjugate with $\Pi^\sim (z) = \Pi^\top (z^{-1})$.
For positive definite matrices $\P\succ 0$ we define the norm $\norm{x}^2_\P=x^\top \P x$.
\section{Setup}\label{sec:setup}
We consider the following linear discrete-time system 
\begin{subequations}\label{eq:sys}
    \begin{align}
    \x_{\t+1}&=\A\x_\t+\Bw\w_\t +\Bd\d_\t+\Bu\u_\t \label{eq:sys_x}\\
    \y_\t &= \C\x_\t + \Dw\w_\t+\Dd\d_\t+ \Du\u_\t \label{eq:y}
    \end{align}
\end{subequations}
with state vector $\x_\t\in\R^\dimx$, control input $\u_\t\in\R^\dimu$, external disturbance\footnote{ Note that this setup includes the special case of two different disturbances $\d^x$ on the state and $\d^y$ the output. In this special case often separate bounds $\dmax^x$ and $\dmax^y$ are known and can be considered to reduce the conservatism.} $\d_\t \in \dset=\{\d\in \R^\dimd|\norm{\d}_\Xi \leq \dmax\}$, $\dmax\geq 0$, $\Xi\succ 0$, uncertainty signal $\w_\t \in \R^\dimw$, output $\y_\t \in\R^{\dimy}$, and the real matrices $\A, \Bw, \Bd, \Bu, \C, \Dw, \Dd, \Du$ with suitable dimensions.
The system is interconnected in feedback with a bounded and causal uncertainty $\Delta:\ltwoe{\dimy}\to\ltwoe{\dimw}$ 
\begin{align}\label{eq:w}
\w_\t = \Delta(\y)_\t,
\end{align}
which is dynamic and depends on the past measurements. Hence, the uncertainty may for example contain unmodeled dynamics, model mismatch, or delays.
Note that the output $\y$ does not denote the vector of measured signals but the vector of signals that enter the uncertainty $\Delta$.

\begin{assumption}[Well-posedness]\label{ass:wp}
    The operator $\Delta$ is bounded and causal and the interconnection of~\eqref{eq:sys} and~\eqref{eq:w} is well-posed, i.e., for each $\d\in\ltwoe{\dimd}$ and $\u\in\ltwoe{\dimu}$ there exists a unique response $\y\in\ltwoe{\dimy}$, $\w\in\ltwoe{\dimw}$, $\x\in\ltwoe{\dimx}$.
\end{assumption}

This assumption guarantees that system \eqref{eq:sys} admits a unique solution, i.e., there is no algebraic loop, which trivially holds in the case $\Dw=0$. 
Considering well-posed interconnections of a known system and an unknown system is a classical robust control setup, e.g., similar to \cite{Hu2016a}.

The control objective is ISS from $\d$ to $\x$ while satisfying the polytopic state and input constraints
\begin{align}\label{eq:constr}
\constrMat \begin{bmatrix}
\x_\t \\ \u_\t
\end{bmatrix} \leq \constrVec
\end{align}
for all times $\t\geq 0$.
The rows of $\constrMat\in\R^{\dimconstrVec\times (\dimx+\dimu)}$ and $\constrVec\in\R^\dimconstrVec$ are denoted by $\constrMat_i$ and $\constrVec_i$ for each $i\in 1,\dots,\dimconstrVec$, respectively.
To keep the theoretical derivations concise and clear, we assume that full state measurement is available.
We base our approach to solve this problem on tube-based MPC which is introduced in the following.

\subsection{Tube-Based Model Predictive Control}\label{sec:tube_mpc}
In this subsection we briefly sketch the idea of tube-based MPC, which was introduced almost simultaneously by~\cite{Chisci2001} and~\cite{Mayne2001} for the case of additive bounded disturbances.
MPC in general is based on predicting the state trajectories with a model and as common in MPC (e.g.~\cite{Kouvaritakis2016}), we denote the predictions at time $\t$ that predict $\k$ steps into the future with the index $\k|\t$. 
In the presence of disturbances and model mismatches, however, precise predictions are impossible and thus, in tube-based MPC a set confining all possible uncertain trajectories is predicted -- the so-called tube.
This tube is centered around a nominal trajectory that follows the uncertainty-free model
\begin{subequations}\label{eq:nom}
    \begin{align}\label{eq:nomx}
    \nomx_{\k+1|\t}&=\A\nomx_{\k|\t}+\Bu\nomu_{\k|\t} \\
    \nomy_{\k|\t} &= \C \nomx_{\k|\t} + \Du \nomu_{\k|\t}\label{eq:nomy}
    \end{align}
\end{subequations}
with nominal prediction $\nomx_{\k|\t}$, nominal input $\nomu_{\k|\t}$, and nominal output $\nomy_{\k|\t}$.
The tube contains all possible trajectories $\x_{\k|\t}$ that follow the true system dynamic $\eqref{eq:sys}$ for $\k\geq 0$ 
\begin{subequations}\label{eq:sys_pred}
    \begin{align}
    \x_{\k+1|\t}&=\A\x_{\k|\t}+\Bw\w_{\k|\t} +\Bd\d_{\k+\t}+\Bu\u_{\k|\t} \label{eq:sys_x_pred}\\
    \y_{\k|\t} &= \C\x_{\k|\t} + \Dw\w_{\k|\t}+\Dd\d_{\k+\t}+ \Du\u_{\k|\t} \\
    \w_{\k|\t} &= \Delta(\y_{\cdot|\t})_\k
    \end{align}
\end{subequations}
starting at the current state $\x_{0|\t}=\x_{\t}$ and having the same past $\y_{-\k|\t}=\y_{\t-\k}$ for $\k\in[1,\t]$.
While we assume full state measurement of $\x_\t$, the disturbances $\d_{\k+\t}$, the uncertainty $\Delta$ and thus $\w_{\k|\t}$ are unknown, and thus $\x_{\k|\t}$ for $\k\geq 1$ is unknown as well.
Hence, the possible future state $\x_{\k|\t}$ is neither a prediction (unknown at time $\t$) nor a realization (we might choose other inputs), it is a \textit{what-if} state meaning where would the state $\x_{\t+\k}$ be if from now (time $\t$) on we apply the inputs $\u_{0|\t},\dots,\u_{\k-1|\t}$ and the external disturbances $\d_{\t},\dots,\d_{\t+\k-1}$.

The error between the possible future state $\x_{\k|\t}$ and the nominal prediction $\nomx_{\k|\t}$ is denoted by $\e_{\k|\t}=\x_{\k|\t}-\nomx_{\k|\t}$.
To ensure that the error $\e_{\k|\t}$ does not diverge, the MPC control action $\nomu_{\k|\t}$ is augmented with a feedback of the error
\begin{align} \label{eq:input}
\u_{\k|\t}=\nomu_{\k|\t} + \K\e_{\k|\t}
\end{align}
where the feedback gain $\K$ is static.
From a robust control point of view it might seem unusual and limiting to consider a static feedback $\K$, however, to keep the derivations concise and clear and to be consistent with tube-based MPC literature we use a static $\K$ in this work, although it might be possible to extend the framework to dynamic controllers $\K$.
Hence, $\u_{\k|\t}$ is a \textit{what-if} input that includes knowledge of the possible future error $\e_{\k|\t}$.
Thereby, the feedback $\K\e_{\k|\t}$ regulates $\x_{\k|\t}$ towards the nominal trajectory $\nomx_{\k|\t}$, while the control action $\nomu_{\k|\t}$ steers the nominal trajectory. 
This key feature of tube-based MPC significantly reduces the conservatism as the feedback $\K$ can keep the tube confining all possible trajectories small by stabilizing the error dynamics
\begin{subequations}\label{eq:error_interconnection}
    \begin{align}\label{eq:error_dyn}
    \e_{\k+1|\t} &= \Ae\e_{\k|\t} + \BGw\w_{\k|\t} + \BGd \d_{\k+\t}\\
    \y_{\k|\t}&= \Ce\e_{\k|\t}  + \DGw\w_{\k|\t}+\DGd\d_{\k+\t} + \nomy_{\k|\t}  \label{eq:y_pred} \\
    \w_{\k|\t} &= \Delta(\y_{\cdot|\t})_\k,
    \end{align}
\end{subequations}
where $\Ae=\A+\Bu\K$ and $\Ce=\C+\Du\K$.
The feedback interconnection~\eqref{eq:error_interconnection} of the error dynamics and the uncertainty $\Delta$ is well-posed\footnote{Well-posedness follows from Ass.~\ref{ass:wp} and the fact that \eqref{eq:sys} and \eqref{eq:error_interconnection} have the same feedthrough matrix $\Dw$.} and in the form shown in Fig.~\ref{fig:feedback_interconnection} with
\begin{align}
    \G=\ssrep{\AG}{\BGw\ \ \BGd}{\CG}{\DGw\ \ \DGd}.
\end{align}

For now, we have considered predictions at a fixed time $\t$.
In closed loop, the MPC controller solves an open-loop finite-horizon optimal control problem to decide on the new nominal initial condition $\nomx_{0|\t}$ and a new nominal input sequence $\nomu_{\cdot |t}$.
When determining  $\nomx_{0|\t}$ and $\nomu_{\cdot |t}$ it must be ensured that the constraints \eqref{eq:constr} are not only satisfied for  $\nomx_{\cdot|\t}$ and $\nomu_{\cdot |t}$, but for the whole tube around this nominal trajectory to ensure robust constraint satisfaction, i.e., these constraints must be tightened according to the size of the tube.
Then, the first input of the control sequence is applied to the system, i.e., $\u_\t = \u_{0|\t}=\nomu_{0|\t}+\K(\x_\t-\nomx_{0|\t})$, which recursively renders $\w_\t = \w_{0|\t}$ and $\x_{\t+1} = \x_{1|\t}$ since $\Delta$ is causal.
Fixing the initial state $\nomx_{0|\t} = \nomx_{1|\t-1}$ to follow the nominal dynamics~\eqref{eq:nomx} as proposed in the early work~\cite{Mayne2001} simplifies the analysis significantly.
This case has been considered in the preliminary conference paper~\cite{Schwenkel2020}.
Nevertheless, at each time $\t$, we obtain a new measurement $\x_\t$ and we want to make use of this new information when initializing the nominal trajectory $\nomx_{0|\t}$.
In~\cite{Chisci2001} it was suggested to initialize $\nomx_{0|\t} = \x_\t$, but it is not guaranteed that this choice is actually better.
Thus,~\cite{Mayne2005} proposed to treat $\nomx_{0|\t}$ as free decision variable and to optimize over all $\nomx_{0|\t}$ that contain the current measurement in the tube centered around them.
In the present work, we want to use this additional degree of freedom since it leads to faster convergence as discussed in~\cite{Rawlings2017} and larger regions of attraction as we will see in our numerical example.
This, however, implies that we need to specify how the error evolves in closed loop, i.e., if the time $\t$ increments and a new nominal initial state $\nomx_{0|\t}$ is chosen
\begin{align}\label{eq:error_dyn_cl}
\e_{0|\t+1} = \e_{1|\t} - \nomx_{0|\t+1} + \nomx_{1|\t}.
\end{align} 
Although tube-based MPC schemes have different definitions of the tube, they are always based on a bound on the error $\e_{\k|\t}$.
A key step in this article is to derive such an error bound based on the assumption that the input-output-behavior of $\Delta$ can be described with a $\rho$-hard IQCs and that the disturbance $\d_\t \in \dset$ is bounded.
Therefore, in the remainder of this section, we give a short introduction into $\rho$-hard IQCs, which provide a general framework to analyze interconnections of the form in Fig.~\ref{fig:feedback_interconnection} for dynamic and static uncertainties $\Delta$.

\subsection{$\rho$-hard Integral Quadratic Constraints}
IQCs originate from the seminal work~\cite{Megretski1997} and are a powerful tool to analyze feedback interconnections as in Fig.~\ref{fig:feedback_interconnection} of a known linear system $G$ and an unknown, possibly nonlinear, operator $\Delta$.
Originally, the framework was developed from a continuous-time frequency-domain point of view, however, the IQC framework has been extended to time-domain formulations via dissipation inequalities~\cite{Seiler2015} and to discrete-time systems~\cite{Hu2016a}.
We build our analysis on the framework of $\rho$-hard IQCs, which were developed in~\cite{Lessard2016}, to analyze exponential stability of discrete-time systems with time-domain IQCs.
This enables us to construct a bound on $\e_{\k|\t}$ in form of an exponentially stable error bounding system.
Let us start by defining a time-domain $\rho$-hard IQC.

\begin{definition}[{$\rho$-hard IQC,~\cite[Definition 3]{Lessard2016}}]
    Let $\rho\in(0,1]$, $M\in\R^{\dimfilty\times \dimfilty}$ and $\filt\in\RHinf^{\dimfilty\times (\dimy+\dimw)}$. A bounded operator $\Delta:\ltwoe{\dimy}\to\ltwoe{\dimw}$ is said to satisfy the $\rho$\emph{-hard IQC} defined by $(\filt,\M)$ if for all $\y\in\ltwoe{\dimy}$ and for all $\T\geq 1$ the following inequality holds
    \begin{align}\label{eq:hardIQC}
    \sum_{\t=0}^{\T-1} \rho^{-2\t}\filty_\t^\top \M\filty_\t\geq 0,\ \text{where } \filty=\filt \begin{bmatrix}
    \y\\\Delta(\y)
    \end{bmatrix}.
    \end{align}
\end{definition}

The key idea when analyzing interconnections as in Fig.~\ref{fig:feedback_interconnection} where $\Delta$ satisfies the $\rho$-hard IQC defined by $(\filt,\M)$ is to replace the uncertain component $\Delta$ with the filter $\filt$ and to consider $\w$ as an input that obeys the output constraint~\eqref{eq:hardIQC}.
This is sketched in Fig.~\ref{fig:augmented_dyn_standard}.
With a state space realization of the filter
\begin{align}\label{eq:filt}
\filt=\left[\begin{array}{c|cc}
\filtA & \filtBin & \filtBout \\ \hline \filtC & \filtDin & \filtDout
\end{array}\right]
\end{align} 
we can write the transfer function $\w\to\filty$ ($\d=0$, $\nomy=0$) as
\begin{align}\label{eq:aug_dyn}
\left[\begin{array}{c|c}
\allA & \allB \\ \hline \allC & \allD
\end{array}\right] = \left[\begin{array}{cc|c}
\AG & 0 & \BGw  \\
\filtBin \CG & \filtA & \filtBin \DGw + \filtBout \\ \hline \filtDin \CG & \filtC & \filtDin \DGw + \filtDout
\end{array}\right].
\end{align}
We will denote the state vector of $\filt$ at time $\t$ with $\filtx_\t$.
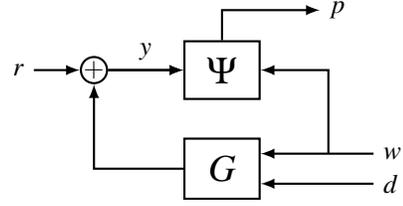
\begin{figure}
  \centering
  \begin{tikzpicture}
    \node [block] (Delta) {\Large $\filt$};
    \node [block, below=of Delta, minimum height=0.8cm] (G) {\Large $\G$};
    \node[sum, left=of Delta] (sum) {};
    \draw[-latex, thick] (G.east) +(1.5,-0.2) node [right] {$\d$} -- ($(G.east)+(0,-0.2)$);
    \draw[-latex, thick] (G.east) +(1.5,0.2) node [right] {$\w$} -- ($(G.east)+(0,0.2)$);
    \draw[latex-, thick] (Delta.east) -| ($(G.east)+(0.9,0.2)$);
    \draw[-latex, thick] (sum) -- node[above] {$\y$} (Delta);
    \draw[latex-, thick] (sum) |- (G);
    \draw[-latex, thick] (sum) --  (Delta);
    \draw[-latex,thick] (Delta.north) |- +(1.3,0.4) node[right] {$\filty$};    
    \draw[-latex, thick] (sum) +(-0.8,0) node [left] {$\nomy$} -- (sum);
  \end{tikzpicture}
  \caption{The IQC characterization allows us to replace $\Delta$ with the filter $\Psi$ and to consider the uncertainty output $\w$ as external input that satisfies the constraint~\eqref{eq:hardIQC} on $\filty$.}\label{fig:augmented_dyn_standard}
\end{figure}

The following assumption is made on the uncertainty.
\begin{assumption}[Uncertainty]\label{ass:iqc}
    The operator $\Delta$ satisfies the $\rho$-hard IQC defined by $(\filt,\M)$ and $\rho\in(0,1]$. 
    The filter $\Psi$ is initialized with $\psi_0=0$.
    There exists $\P\succ 0$ such that the following matrix inequality holds
    \begin{align}\label{eq:lmi}
    \begin{bmatrix}
    I & 0\\
    \allA & \allB \\ \allC & \allD 
    \end{bmatrix}^\top \begin{bmatrix}
    -\rho^2\P & 0 & 0 \\
    0 & \P & 0 \\
    0 & 0 & \M
    \end{bmatrix}\begin{bmatrix}
    I & 0\\
    \allA & \allB \\ \allC & \allD 
    \end{bmatrix}\prec 0.
    \end{align}
\end{assumption}

Note that in the tube-based MPC setting from Sec.~\ref{sec:tube_mpc}, the matrices $\AG$ and $\CG$ and thus also $\allA$ and $\allC$ depend on the feedback gain $\K$. 
Throughout this work, we assume that a suitable $\K$ satisfying Ass.~\ref{ass:iqc} is given. 
For a fixed gain $\K$ and a fixed constant $\rho$, \eqref{eq:lmi} reduces to a linear matrix inequality (LMI), which can thus be embedded in a suitable offline computation of the matrix $\P$ (cf. Rem.~\ref{rem:optimize_P} and~\ref{rem:pseudoAlgorithm}).
Based on Ass.~\ref{ass:iqc}, the following exponential stability bound was shown in~\cite{Lessard2016}.

\begin{theorem}[{Exponential stability,~\cite[Thm.~4\footnotemark]{Lessard2016}}]\label{thm:iqc}
    Let Ass.~\ref{ass:wp},~\ref{ass:iqc}, and $(r,d)=0$ hold. Then system~\eqref{eq:error_interconnection} is $\rho$-exponentially stable, i.e., $\norm{\e_\t} \leq \sqrt{\mathrm{cond}(\P)} \rho^{\t} \norm{\e_0}$ for all $\t\geq 0$.
\end{theorem}
\footnotetext{To be precise, the theorem is slightly altered compared to~\cite{Lessard2016}: First,~\cite{Lessard2016} requires only semi definiteness of the LMI~\eqref{eq:lmi}, i.e.,  $\preceq 0$. And second, in~\cite{Lessard2016} the theorem is stated for $\Dw=0$. However, the result still holds due to the well-posedness assumption; the proof is analogous.}

Notice that the authors of~\cite{Lessard2016} analyzed this interconnection with $\d=0$ and $\nomy=0$.
Hence, this result cannot immediately be applied to our setup, as we have $\d\neq 0$ and $\nomy\neq 0$.
However, Lemma~\ref{lem:error_bound} in Sec.~\ref{sec:error_sys} below extends this result under the same Ass.~\ref{ass:iqc} to $\d\neq 0$ and $\nomy\neq 0$ and additionally includes an estimate to account for the nominal initial state optimization~\eqref{eq:error_dyn_cl}.

In the remainder of this section, we provide more insights into $\rho$-hard IQCs by bridging the gap from the time-domain perspective to the frequency-domain framework of $\rho$-IQCs~\cite{Boczar2015}, whereas the error bounds relevant for the MPC are developed in Sec.~\ref{sec:error_sys}.
In particular, we will see that a general class of frequency-domain $\rho$-IQCs can be equivalently formulated as time-domain $\rho$-hard IQCs. 

\begin{definition}[{$\rho$-IQC,~\cite[Definition 6]{Boczar2015}}]
    Let $\rho\in(0,1]$ and $\Pi=\Pi^*\in\RLinf^{(\dimy+\dimw)\times (\dimy+\dimw)}$. A bounded operator $\Delta:\ltwoe{\dimy}\to\ltwoe{\dimw}$ is said to satisfy the $\rho$\emph{-IQC defined by the multiplier }$\Pi$ if for all $\y\in\ltworho{\dimy}$ and $\w=\Delta(\y)$ the following inequality holds\\[-0.6cm]
    \begin{align}\label{eq:iqc_fd}
        \int_{\mathbb T} \begin{bmatrix}
            \hat \y (\rho z) \\ \hat \w(\rho z)
        \end{bmatrix}^* \Pi(\rho z) \begin{bmatrix}
            \hat \y (\rho z) \\ \hat \w(\rho z)
        \end{bmatrix} \dint z \geq 0.
    \end{align}
\end{definition}

There is also an exponential stability result with $\rho$-IQCs formulated in the frequency domain.

\begin{theorem}[{Exponential Stability,~\cite[Thm.~8]{Boczar2015}}]\label{thm:iqc_fd}
    Let $\rho\in(0,1)$, $\G_\rho \in \RHinf$ and $\Delta$ be a bounded causal operator. Suppose that for all $\tau\in[0,1]$:
    \begin{enumerate}
        \item the interconnection of $G$ and $\tau\Delta$ is  well-posed
        \item $\tau\Delta$ satisfies the $\rho$-IQC defined by $\Pi$
        \item there exists $\varepsilon>0$ such that\\[-0.6cm]
        \begin{align}\label{eq:fdi}
            \begin{bmatrix} \G(\rho z) \\ I \end{bmatrix}^*\Pi(\rho z) 
            \begin{bmatrix} \G(\rho z) \\ I \end{bmatrix} 
            \preceq -\varepsilon I, \quad  \forall z \in \mathbb T.
        \end{align}
    \end{enumerate}
    Then the interconnection of $\G$ and $\Delta$ as shown in Fig.~\ref{fig:feedback_interconnection} is exponentially stable with rate $\rho$ for $\nomy=0$, $\d=0$.
\end{theorem}

In order to compare Thm.~\ref{thm:iqc} and Thm.~\ref{thm:iqc_fd}, we note that each frequency-domain $\rho$-IQC can be related to time domain by a factorization of the multiplier $\Pi$.

\begin{definition}[$\rho$-factorization] 
Let $\filt \in \RLinf$, $\Pi\in\RLinf$, and $\M\in\R^{\dimfilty\times\dimfilty}$. We call $(\filt,\M)$ a $\rho$\emph{-factorization} of $\Pi$ if $\Pi_\rho=(\filt_\rho)^\sim\M\filt_\rho$ and $\filt_\rho\in\RHinf$.
\end{definition}

If $(\filt,\M)$ is a $\rho$-factorization of $\Pi$, then, as shown in~\cite[Rem.~10]{Boczar2015} by applying Parseval's Theorem, the $\rho$-IQC defined by $\Pi$ is satisfied if and only if~\eqref{eq:hardIQC} holds for $\T=\infty$.
Hence, $\rho$-hard IQCs imply $\rho$-IQCs while the opposite is in general not true.
Further, it was shown in~\cite[Corollary 12]{Boczar2015} that~\eqref{eq:fdi} is equivalent to the existence of $\P=\P^\top$ satisfying~\eqref{eq:lmi}.
Note that again, the time-domain requirement $P\succ 0$ is stricter.
This leads to the interesting question for which multipliers the frequency-domain $\rho$-IQC also implies a $\rho$-hard IQC in the time domain and for which multipliers $\P$ is guaranteed to be positive definite.
If $\rho=1$, it is known that so-called strict Positive Negative (PN) multipliers have both properties~\cite{Hu2016a}.

As defined in~\cite[Definition~4]{Hu2016a}, strict PN-multipliers are multipliers $\Pi=\Pi^\sim\in\RLinf$ where the first block diagonal entry with dimension $\dimy\times\dimy$ is positive definite and the second block diagonal entry with dimension $\dimw\times\dimw$ is negative definite for all frequencies $z\in\mathbb T$.
The following theorem extends the results of~\cite{Hu2016a} to $\rho$-IQCs with general $\rho\in(0,1]$ and show that strict PN multipliers admit $\rho$-hard factorizations and that~\eqref{eq:fdi} leads to~\eqref{eq:lmi} with $P\succ 0$.

\begin{theorem}[From $\rho$-IQCs to $\rho$-hard IQCs]\label{thm:pn-hard}
    Let $\rho \in (0,1]$, $G_\rho\in\RHinf$, $\Pi\in\RLinf$, and $\Pi_\rho$ be a strict PN multiplier. Then there exists a $\rho$-factorization $(\filt,\M)$ of $\Pi$ such that all $\Delta$ satisfying the $\rho$-IQC defined by $\Pi$ also satisfy the $\rho$-hard IQC defined by $(\filt,M)$. Further, if~\eqref{eq:fdi} holds, then there exists a $\P\succ 0$ such that~\eqref{eq:lmi} holds.
\end{theorem}
A proof of this result can be found in Appendix~\ref{sec:proof_thm_2}.
The theorem shows that a large class of frequency-domain $\rho$-IQC multipliers have a $\rho$-hard time-domain factorization. 
This is especially helpful since many IQCs are more conveniently derived in the frequency domain.

\section{Exponentially Stable Error Bounding System}\label{sec:error_sys}

In this section, we derive a bound on the error $\e_{\k|\t}$ based on Ass.~\ref{ass:iqc} by making use of the $\rho$-hard IQC that bounds $\Delta$.
Instead of a static error bound as in~\cite{Falugi2014}, we reduce conservatism with a dynamic error bound in form of a scalar exponentially stable system with the inputs $\d$ and $\nomy$, i.e., that depends on the disturbance and the nominal excitation of the uncertainty.
This idea is inspired by the work \cite{Loehning2014}, where an error bounding system was used to describe the deviation of model order reductions.

First, let us introduce the notation $\filtx_{\k|\t}$ for the possible future trajectory of the state of filter $\filt$ from \eqref{eq:filt}
\begin{subequations}\label{eq:filt_ss}
    \begin{align}
    \filtx_{\k+1|\t} &= \filtA\filtx_{\k|\t}+\filtBin\y_{\k|\t}+\filtBout\w_{\k|\t}\\
    \filty_{\k|\t} &= \filtC\filtx_{\k|\t}+\filtDin\y_{\k|\t}+\filtDout\w_{\k|\t}
    \end{align}
\end{subequations}
 with $\filtx_{0|\t+1}=\filtx_{1|\t}$ and $\filtx_{0|0}=\filtx_0$.
Note that this guarantees $\filtx_{0|\t}=\filtx_\t$ and $\filty_{0|\t}=\filty_\t$ since $\w_{0|\t}=\w_\t$ and $\y_{0|\t}=\y_\t$ as discussed above.
Second, with the help of this notation and based on Ass.~\ref{ass:iqc}, we can bound the discrepancy $\e_{\k|\t}$ between the nominal state $\nomx_{\k|\t}$ and the possible future state $\x_{\k|\t}$ as shown in the following lemma.

\begin{lemma}[Exponentially stable error bounding system]\label{lem:error_bound}
        Consider the interconnection \eqref{eq:error_interconnection} and let Ass.~\ref{ass:iqc} hold. 
        Then, there exist $\Gamma\in\R^{\dimy\times\dimy}$, $\Gamma\succ 0$ and $\gamma>0$ satisfying the following LMI
        \begin{align}\label{eq:finsler_lmi}
        \newcommand{\colspace}{\!\!\!}\colspace
        \begin{bmatrix}
        \\\\ \symLMI \\\\\\
        \end{bmatrix}^\top\!\! \begin{bmatrix}
        -\rho^2\P & \colspace 0 & \colspace 0 & \colspace 0 \\
        0 & \colspace\P & \colspace 0 & \colspace 0\\
        0 & \colspace 0 & \colspace\M & \colspace 0 \\ 0& \colspace 0&\colspace 0&\colspace\!- \Lambda
        \end{bmatrix}\begin{bmatrix}
        I & \colspace 0 & 0\\
        \allA & \colspace\allB & \colspace\begin{bmatrix} \BGd & \colspace 0 \\ \filtBin\DGd &\colspace \filtBin \end{bmatrix}\\ \allC &\colspace \allD & \colspace\begin{bmatrix}    \filtDin\DGd & \colspace\filtDin \end{bmatrix} \\ 0&\colspace 0&I
        \end{bmatrix}\prec 0
        \end{align} 
        with $\Lambda=\diag(\gamma \Xi,\Gamma)$.
        Further, for any sequences $\d \in \ltwoe{\dimd}$, $\left(\nomy_{\cdot|\t}\right)_{\t\in\mathbb{N}}$, $\nomy_{\cdot|\t} \in \ltwoe{\dimy}$, and $\e_{0|\cdot}\in\ltwoe{\dimx}$ the following inequality holds for all times $\t\geq 0$ and all predictions $\k\geq 0$
    \begin{align}\label{eq:error_bound_ineq}
     \norm{\begin{bmatrix}
        \e_{\k|\t} \\
        \filtx_{\k|\t}
        \end{bmatrix}}_\P^2\leq \c_{\k|\t},
    \end{align} with 
    $\c_{0|0}=\norm{\big[\begin{matrix} \e_{0|0}^\top&0 \end{matrix}\big]^\top}_{\P}^2$ and $\c_{\k|\t}$ recursively defined by
    \begin{subequations}\label{eq:error_bound}
        \begin{align}
        \c_{\k+1|\t} &= \rho^2\c_{\k|\t} +\gamma\norm{\d_{\t+\k}}^2_{\Xi}+ \norm{\nomy_{\k|\t}}^2_{\Gamma},\label{eq:c_ol}\\
        \c_{0|\t+1} &= \c_{1|\t} +\norm{\!\begin{bmatrix}
            \e_{0|\t+1} \\ \filtx_{0|\t+1}
            \end{bmatrix}\!}_\P^2 - \norm{\!\begin{bmatrix}
            \e_{1|\t} \\ \filtx_{1|\t}
            \end{bmatrix}\!}_\P^2. \label{eq:c_cl}
        \end{align}
    \end{subequations}
\end{lemma}

\begin{proof}
    First, we derive~\eqref{eq:finsler_lmi} from~\eqref{eq:lmi} by using Finsler's Lemma; second, we use~\eqref{eq:finsler_lmi} to show that a dissipation inequality holds; and third, summing up this inequality from $0$ to $\t$ yields~\eqref{eq:error_bound_ineq} and~\eqref{eq:error_bound}.
    
    The LMI~\eqref{eq:lmi} guarantees that~\eqref{eq:finsler_lmi} holds whenever multiplied from left with $\begin{bmatrix}
        \e^\top_\k & \filtx^\top_\k & \w^\top_\k & \d^\top_\k & \nomy_\k^\top
    \end{bmatrix}$ and from right with its transpose for $\d_\k=0$, $\nomy_\k=0$.
    Thus, Finsler's Lemma \cite{Finsler1936} guarantees the existence of $\bar \gamma>0$ large enough such that~\eqref{eq:finsler_lmi} holds with $\Lambda =\bar \gamma I$ also for nonzero $\nomy_\k$, $\d_\k$.
    Therefore, we know that any $\Gamma\succeq\bar \gamma I$ and $\gamma \Xi \succeq \bar \gamma I$ satisfy~\eqref{eq:finsler_lmi} with $\Lambda=\diag(\Gamma,\gamma \Xi)\succeq \bar \gamma I$ as well.
    
    For the second part, we use
    \newcommand{\colspace}{\!\!\!}
    \begin{align*}
    &\begin{bmatrix}
    \allA & \colspace\allB & \colspace\begin{bmatrix} \BGd & \colspace 0 \\ \filtBin\DGd &\colspace \filtBin \end{bmatrix}\\[0.25cm] \allC &\colspace \allD & \colspace\begin{bmatrix}    \filtDin\DGd & \colspace\filtDin \end{bmatrix}
    \end{bmatrix}\begin{bmatrix}
    \e_{\kappa|\tau} \\ \filtx_{\kappa|\tau} \\ \w_{\kappa|\tau} \\ \d_{\kappa+\tau} \\ \nomy_{\kappa|\tau}
    \end{bmatrix} \\
    &\quad\refeq{(\ref{eq:y_pred}, \ref{eq:aug_dyn})}{=}\ \ \begin{bmatrix}
    \Ae\e_{\kappa|\tau}+\Bw\w_{\kappa|\tau} +\BGd\d_{\kappa+\tau} \\ 
    \filtA \filtx_{\kappa|\tau} + \filtBin \y_{\kappa|\tau} + \filtBout  \w_{\kappa|\tau} \\ 
    \filtC \filtx_{\kappa|\tau} + \filtDin \y_{\kappa|\tau} + \filtDout  \w_{\kappa|\tau} 
    \end{bmatrix}\ \ \refeq{(\ref{eq:error_dyn}, \ref{eq:filt_ss})}{=}\ \     \begin{bmatrix}\e_{\kappa+1|\tau} \\ \filtx_{\kappa+1|\tau} \\ \filty_{\kappa|\tau}
    \end{bmatrix}
    \end{align*}
    when multiplying \eqref{eq:finsler_lmi} from the left with $\begin{bmatrix}
      \e_{\kappa|\tau}^\top & \filtx_{\kappa|\tau}^\top & \w_{\kappa|\tau}^\top & \d_{\kappa+\tau}^\top & \nomy_{\kappa|\tau}^\top
      \end{bmatrix}$ and from the right with its transpose, which leads to the dissipation inequality
    \begin{align}\label{eq:diss-ineq}
     \begin{split}
     &\norm{\begin{bmatrix}
     \e_{\kappa+1|\tau} \\
     \filtx_{\kappa+1|\tau}
     \end{bmatrix}}_\P^2 - \rho^2 
     \norm{\begin{bmatrix}
     \e_{\kappa|\tau} \\
     \filtx_{\kappa|\tau}
     \end{bmatrix}}_\P^2 
     \\ & \qquad \quad + \filty^\top_{\kappa|\tau} \M \filty_{\kappa|\tau}-\gamma \norm{ \d_{\kappa+\tau}}_\Xi^2- \norm{\nomy_{\kappa|\tau}}^2_\Gamma\leq 0.
     \end{split}
    \end{align}
    In order to get rid of the unknown $\filty_{\kappa|\tau}$, we utilize the IQC~\eqref{eq:hardIQC} by multiplying \eqref{eq:diss-ineq} with suitable powers of $\rho^2$ and sum it up over the past from $0$ to $\t$ (for $\kappa =0 $) and over the predictions from $0|\t$ to $\k|\t$.
    Let us start with $0|0$, we plug $\c_{0|0}=\norm{\big[\begin{matrix}
            \e_{0|0}^\top&\filtx_{0|0}
        \end{matrix}\big]^\top}_{\P}^2$ and~\eqref{eq:c_ol} in~\eqref{eq:diss-ineq} to obtain
    \begin{align*}
        0\geq \norm{\begin{bmatrix}
                \e_{1|0} \\
                \filtx_{1|0}
        \end{bmatrix}}_\P^2+ \filty^\top_{0|0} \M \filty_{0|0}-\c_{1|0} =: \Sigma_1
    \end{align*}
    As next part we consider $0|1$ to $0|\t$, i.e., $\kappa=0$, $\tau \in [1,\t]$. 
    In this case, we can use~\eqref{eq:error_bound} to rewrite~\eqref{eq:diss-ineq} in this case as
    \begin{align}\label{eq:diss-ineq2}
        \begin{split}
            &\norm{\begin{bmatrix}
                    \e_{1|\tau} \\
                    \filtx_{1|\tau}
            \end{bmatrix}}_\P^2 - \rho^2 
            \norm{\begin{bmatrix}
                    \e_{1|\tau-1} \\
                    \filtx_{1|\tau-1}
            \end{bmatrix}}_\P^2 
            \\ & \qquad \quad + \filty^\top_{0|\tau} \M \filty_{0|\tau}-\c_{1|\tau}+\rho^2\c_{1|\tau-1}\leq 0.
        \end{split}
    \end{align}
    Now we use a telescoping sum argument by multiplying~\eqref{eq:diss-ineq2} with $\rho^{2(t-\tau)}$ and sum over $\tau$ from $1$ to $\t$:
    \begin{align*}
        0\refeq{\eqref{eq:diss-ineq2}}{\geq}\   &\norm{\begin{bmatrix}
                \e_{1|\t} \\
                \filtx_{1|\t}
        \end{bmatrix}}_\P^2 -\rho^{2\t} \norm{\begin{bmatrix}\e_{1|0} \\ \filtx_{1|0} \end{bmatrix}}_{\P}^2 \\
        & + \sum_{\tau=1}^{\t}\rho^{2(\t-\tau)} \filty^\top_{0|\tau} \M \filty_{0|\tau}
        -\c_{1|\t}+\rho^{2\t} \c_{1|0} =:\Sigma_2
    \end{align*}
    As third part, we consider $1|\t$ to $\k|\t$, i.e.,  $\kappa$ from $1$ to $\k-1$ for $\tau=\t$. In this case, we can use~\eqref{eq:c_ol} to rewrite~\eqref{eq:diss-ineq} as
    \begin{align}\label{eq:diss-ineq3}
        \begin{split}
            &\norm{\begin{bmatrix}
                    \e_{\kappa+1|\t} \\
                    \filtx_{\kappa + 1|\t}
            \end{bmatrix}}_\P^2 - \rho^2 
            \norm{\begin{bmatrix}
                    \e_{\kappa |\t} \\
                    \filtx_{\kappa|\t}
            \end{bmatrix}}_\P^2 
            \\ & \qquad \quad + \filty^\top_{\kappa |\t} \M \filty_{\kappa|\t}-\c_{\kappa+1|\t}+\rho^2\c_{\kappa|\t}\leq 0.
        \end{split}
    \end{align}
    Now we use a telescoping sum argument by multiplying~\eqref{eq:diss-ineq3} with $\rho^{2(k-\kappa-1)}$ and sum over $\kappa$ from $1$ to $\k-1$:
    \begin{align*}
        0&\refeq{\eqref{eq:diss-ineq3}}{\geq}\, \norm{\begin{bmatrix}
                \e_{\k|\t} \\
                \filtx_{\k|\t}
        \end{bmatrix}}_\P^2 -\rho^{2(\k-1)} \norm{\begin{bmatrix}\e_{1|\t} \\ \filtx_{1|\t} \end{bmatrix}}_{\P}^2 \\
    &\quad+ \sum_{\kappa=1}^{\k-1}\rho^{2(\k-\kappa-1)}\filty^\top_{\kappa|\t} \M \filty_{\kappa|\t}-\c_{\k|\t}+\rho^{2(\k-1)}\c_{1|\t} =:\Sigma_3.
    \end{align*}
    Finally, we sum up $\Sigma_1$, $\Sigma_2$ and $\Sigma_3$ with suitable factors of $\rho^{2}$. Since the sequence $\hat\filty_i:=\filty_{0|\tau}$ for $i=\tau=0,...,\t$ appended with $\hat\filty_i=\filty_{\kappa|\t}$ for $i-\t=\kappa=1,...,\k-1$ is a feasible filter output trajectory, the $\rho$-hard IQC~\eqref{eq:hardIQC} holds for this sequence and we can conclude
    \begin{align*}
        0 &\geq \rho^{2(\t+\k-1)}\Sigma_1 +\rho^{2(\k-1)}\Sigma_2 + \Sigma_{3}\\
        &=\norm{\begin{bmatrix}
                \e_{\k|\t} \\
                \filtx_{\k|\t}
        \end{bmatrix}}_\P^2 -\c_{\k|\t}+ \rho^{2(\t+\k-1)}\sum_{i=0}^{\t+\k-1}\rho^{-2i}\hat \filty^\top_{i} \M \hat \filty_{i}\\
         &\refeq{\eqref{eq:hardIQC}}{\geq } \ \ \norm{\begin{bmatrix}
                \e_{\k|\t} \\
                \filtx_{\k|\t}
        \end{bmatrix}}_\P^2 -\c_{\k|\t}.\\[-1cm]
    \end{align*}
\end{proof}
\begin{remark}
    Lemma~\ref{lem:error_bound} is not only relevant for considering MPC schemes but also for general $\rho$-hard IQC theory.
    In particular the error bound \eqref{eq:error_bound_ineq} and the error bound dynamics \eqref{eq:c_ol} for all $\k\geq 0$ and fixed $\t=0$ might be of interest to other settings considering the interconnection in Fig.~\ref{fig:feedback_interconnection} as for example reachability analysis with IQCs (compare \cite{Yin2020}).
    Then, these equations provide a bound on the state $\e_{\k|0}$ of system $\G$ for nonzero initial conditions $\e_{0|0}$ and nonzero external inputs $\nomy_{k|0}$ and $\d_{\k}$. 
\end{remark}

Unfortunately, we cannot use the bound~\eqref{eq:error_bound_ineq} of Lemma~\ref{lem:error_bound} for constraint tightening in an MPC scheme, since it depends on the generally unknown future disturbances $\d_{\t+\k}$ in the recursion~\eqref{eq:c_ol} and the generally unknown filter state $\filtx_{0|\t}=\filtx_{1|\t-1}$ in the recursion~\eqref{eq:c_cl}. 
Therefore, as a third step, we introduce a known upper bound $\s_{\k|\t}$ on $\c_{\k|\t}$.
\begin{theorem}[Tube dynamics]\label{thm:error_bound}
    Consider the interconnection \eqref{eq:error_interconnection}, let Ass.~\ref{ass:iqc} hold.
    Further, let $\P$ be decomposed into $\P=\begin{bmatrix}
    \P_{11} & \P_{21}^\top \\ \P_{21} & \P_{22}
    \end{bmatrix}$ with $\P_{11}\in\R^{\dimx\times \dimx}$ and define $\Ptube = \P_{11} -\P_{21}^\top \P_{22}^{-1} \P_{21}\succ0$ and $\P_\mathrm{diff}=\P_{11}-\Ptube\succeq 0$.
    Then, there exist $\Gamma\succ 0$ and $\gamma >0$ satisfying~\eqref{eq:finsler_lmi} and the following inequality holds for any sequences $\d \in \ltwoe{\dimd}$, $\nomx_{0|\cdot}\in\ltwoe{\dimx}$, $\left(\nomu_{\cdot|\t}\right)_{\t\in\mathbb{N}}$, $\nomu_{\cdot|\t} \in \ltwoe{\dimu}$, and any initial condition $\x_0\in\R^\dimx$
    \begin{align}\label{eq:error_bound_ineq_mpc}
    \norm{\e_{\k|\t} }_\Ptube^2\leq \s_{\k|\t},
    \end{align} where $\s_{\k|\t}$ is recursively defined by
    \begin{subequations}\label{eq:error_bound_mpc}
        \begin{align}\nonumber
         &\s_{0|\t}= \s_{1|\t-1}+\norm{\e_{0|\t}}_{\Ptube}^2-\norm{\e_{1|\t-1}}_{\Ptube}^2+\norm{\e_{0|\t}-\e_{1|\t-1}}_{\P_\mathrm{diff}}^2\\
        &\quad\qquad +2\norm{\e_{0|\t}-\e_{1|\t-1}}_{\P_\mathrm{diff}}\sqrt{\s_{1|\t-1}-\norm{\e_{1|\t-1}}_{\Ptube}^2} \label{eq:s_cl}\\
        &\s_{\k+1|\t} = \rho^2\s_{\k|\t} +\gamma\dmax^2+ \norm{\nomy_{\k|\t}}^2_{\Gamma},\label{eq:s_ol}
        \end{align}
    \end{subequations}
    with $\s_{0|0}=\norm{\e_{0|0}}_{\P_{11}}^2$.
\end{theorem}
\begin{proof}
    Recall Lemma~\ref{lem:error_bound} which guarantees the existence of $\Gamma\succ 0$ and $\gamma >0$ satisfying~\eqref{eq:finsler_lmi}.
    First, we note that 
    \begin{align}\label{eq:Ptube}
    \begin{bmatrix}
    \Ptube & 0\\0&0
    \end{bmatrix}=\P - \begin{bmatrix}
    \P_{21}^\top \P_{22}^{-1} \P_{21}& \P_{21}^\top \\ \P_{21} & \P_{22}
    \end{bmatrix}\preceq \P
    \end{align}
    since the matrix $\begin{bmatrix}
    \P_{21}^\top \P_{22}^{-1} \P_{21}& \P_{21}^\top \\ \P_{21} & \P_{22}
    \end{bmatrix}$ is positive semi-definite, which can be seen by looking at its Schur complement $\P_{21}^\top \P_{22}^{-1} \P_{21}-\P_{21}^\top \P_{22}^{-1} \P_{21}=0\succeq 0$.
    In view of Lemma~\ref{lem:error_bound}, we can infer
    \begin{align*}
    \norm{\e_{\k|\t} }_\Ptube^2\leq \norm{\begin{bmatrix}
        \e_{\k|\t} \\
        \filtx_{\k|\t}
        \end{bmatrix}}_\P^2\leq \c_{\k|\t}.
    \end{align*}
    In order to prove~\eqref{eq:error_bound_ineq_mpc}, we show $\c_{\k|\t}\leq \s_{\k|\t}$ using a proof of induction. 
    The induction basis is trivial, since $\s_{0|0}=\c_{0|0}$.
    Now we have to do two induction steps, one from $\k|\t$ to $\k+1|\t$ (for $\t\geq 0$) and one from $1|\t-1$ to $0|\t$ (for $\t\geq 1$).
    Let us start with the former by using the induction hypothesis (IH) $\c_{\k|\t}\leq \s_{\k|\t}$ and $\norm{\d_{\t+\k}}_\Xi^2\leq\dmax^2$ in \eqref{eq:s_ol}:
    \begin{align*}
    \s_{\k+1|\t}\, &\refeq{\eqref{eq:s_ol}}{\geq}\, \rho^2\c_{\k|\t} +\gamma\norm{\d_{\t+\k}}_\Xi^2+ \norm{\nomy_{\k|\t}}^2_{\Gamma}\ \refeq{\eqref{eq:c_ol}}{=}\ \c_{\k+1|\t}.
    \end{align*}
    In order to take the step from $1|\t-1$ to $0|\t$, we show that $\s_{0|\t}-\s_{1|\t-1}\geq \c_{0|\t}-\c_{1|\t-1}$, which then implies $\c_{0|\t}\leq \s_{0|\t}$ by~(IH).
    As stated in~\eqref{eq:c_cl}, $\c_{0|\t}-\c_{1|\t-1}$ depends on $\filtx_{1|\t-1}$, thus we maximize it over all possible $\filtx_{1|\t-1}$, i.e., all that satisfy
    \begin{align}\label{eq:filtx_bound_ol_pred}
    \norm{\begin{bmatrix}
        \e_{1|\t-1} \\
        \filtx_{1|\t-1}
        \end{bmatrix}}_\P^2\ \ \refeq{ \eqref{eq:error_bound_ineq}, \text{IH}}{\leq}\ \ \s_{1|\t-1}.
    \end{align}
    Hence, we solve
    \begin{align}\label{eq:max_filtx}
        \c_{0|\t}-\c_{1|\t-1} \leq 
        \max_{\substack{\filtx_{1|\t-1} \\ \text{s.t. \eqref{eq:filtx_bound_ol_pred}}}}&
        \norm{\!\begin{bmatrix}
        \e_{0|\t} \\ \filtx_{1|\t-1}
        \end{bmatrix}\!}_\P^2 - \norm{\!\begin{bmatrix}
        \e_{1|\t-1} \\ \filtx_{1|\t-1}
        \end{bmatrix}\!}_\P^2.
    \end{align}
    To this end, we transform \eqref{eq:max_filtx} to an easier form.
    First, we apply the coordinate shift $\bar \filtx = \filtx_{1|\t-1}+\P_{22}^{-1}\P_{21} \e_{1|\t-1}$, which transforms the constraint~\eqref{eq:filtx_bound_ol_pred} to
    \begin{align*}
        \s_{1|\t-1}&\geq \norm{\begin{bmatrix} \e_{1|\t-1} \\ \filtx_{1|\t-1} \end{bmatrix}}_\P^2 
        \\&= \norm{\e_{1|\t-1}}_{\P_{11}}^2 + 2\e_{1|\t-1}^\top \P_{21}^\top(\bar \filtx-\P_{22}^{-1}\P_{21} \e_{1|\t-1}) \\&\qquad + \norm{\bar \filtx-\P_{22}^{-1}\P_{21} \e_{1|\t-1}}_{\P_{22}}^2\\
        &= \norm{\e_{1|\t-1}}_{\Ptube}^2 + \norm{\bar \filtx}_{\P_{22}}^2
    \end{align*}
    and the objective of~\eqref{eq:max_filtx} to
    \begin{align*}
        &\norm{\!\begin{bmatrix}
                \e_{0|\t} \\ \filtx_{1|\t-1}
            \end{bmatrix}\!}_\P^2 - \norm{\!\begin{bmatrix}
                \e_{1|\t-1} \\ \filtx_{1|\t-1}
            \end{bmatrix}\!}_\P^2=\norm{\e_{0|\t}}_{\P_{11}}^2-\norm{\e_{1|\t-1}}_{\P_{11}}^2\\&\qquad\qquad\qquad+2(\bar \filtx - \P_{22}^{-1}\P_{21} \e_{1|\t-1})^\top \P_{21} (\e_{0|\t}-\e_{1|\t-1})\\
        &=\norm{\e_{0|\t}}_{\P_{11}}^2-\norm{\e_{1|\t-1}}_{\P_{11}}^2-2\e_{1|\t-1}^\top\P_\mathrm{diff} (\e_{0|\t}-\e_{1|\t-1}) \\
        &\quad+2\bar \filtx^\top \P_{21} (\e_{0|\t}-\e_{1|\t-1})\\
        &=\norm{\e_{0|\t}}_{\Ptube}^2-\norm{\e_{1|\t-1}}_{\Ptube}^2+ \norm{\e_{0|\t}-\e_{1|\t-1}}_{\P_\mathrm{diff}}^2 \\&\quad+2\bar \filtx^\top \P_{21} (\e_{0|\t}-\e_{1|\t-1}).
    \end{align*}
    Second, we apply the linear transformation $\tilde \filtx = \P_{22}^{\sfrac{1}{2}}\bar \filtx$ and obtain that the optimization problem~\eqref{eq:max_filtx} is equivalent to
    \begin{align}\label{eq:max_filtx_3}
    \begin{split}
    \max_{\tilde \filtx}&\norm{\e_{0|\t}}_{\Ptube}^2-\norm{\e_{1|\t-1}}_{\Ptube}^2+\norm{\e_{0|\t}-\e_{1|\t-1}}_{\P_\mathrm{diff}}^2\\&+2\tilde \filtx^\top \P_{22}^{-\sfrac{1}{2}\top} \P_{21} (\e_{0|\t}-\e_{1|\t-1}) \\ \text{s.t.\ }&\norm{\tilde \filtx}^2 \leq \s_{1|\t-1}-\norm{\e_{1|\t-1}}_{\Ptube}^2.
    \end{split}
    \end{align}
    This optimization problem has an affine objective and the constraint set is a scaled unit ball.
    Thus the analytical maximum is attained at $\norm{\tilde \filtx}^2 = \s_{1|\t-1}-\norm{\e_{1|\t-1}}_{\Ptube}^2$ with $\tilde \filtx$ pointing in the direction of $\P_{22}^{-\sfrac{1}{2}\top} \P_{21} (\e_{0|\t}-\e_{1|\t-1})$, i.e., the maximum of  \eqref{eq:max_filtx} is
    \begin{align*}
    &\norm{\e_{0|\t}}_{\Ptube}^2-\norm{\e_{1|\t-1}}_{\Ptube}^2+\norm{\e_{0|\t}-\e_{1|\t-1}}_{\P_\mathrm{diff}}^2\\&\quad +2\norm{\e_{0|\t}-\e_{1|\t-1}}_{\P_\mathrm{diff}}\sqrt{\s_{1|\t-1}-\norm{\e_{1|\t-1}}_{\Ptube}^2}.
    \end{align*}
    By definition~\eqref{eq:s_cl}, this is $\s_{0|\t} - \s_{1|\t-1}$.
    To conclude, we have shown that $\s_{0|\t} - \s_{1|\t-1}\geq \c_{0|\t} - \c_{1|\t-1}$ which implies by (IH) $\s_{0|\t} \geq \c_{0|\t}$ and thus completes the proof by induction.
\end{proof}

\section{Proposed MPC scheme}\label{sec:mpc}
In this section, we propose an MPC scheme that handles dynamic uncertainties by using the error bounding system from Thm.~\ref{thm:error_bound} as tube dynamics, which ensures that the tube confines all possible trajectories.
In order to ensure constraint satisfaction of the true but unknown system, we have to choose the nominal inputs $\nomu_{\cdot|\t}$ such that the whole tube around the nominal trajectory is feasible.
\begin{lemma}[Constraint tightening]\label{lem:constr_satisf}
    Consider the interconnection \eqref{eq:nom}--\eqref{eq:error_interconnection} and the tube dynamics \eqref{eq:error_bound_mpc}. Let Ass.~\ref{ass:iqc} hold. 
    Then, the constraints $\constrMat \begin{bmatrix}
    \x_{\k|\t} \\ \u_{\k|\t}
    \end{bmatrix}\leq \constrVec$ hold whenever
    \begin{align}\label{eq:mpc_constr}
    \constrMat \begin{bmatrix}
    \nomx_{\k|\t} \\ \nomu_{\k|\t}
    \end{bmatrix}\leq \constrVec-\sqrt{\s_{\k|\t}}\constrTight
    \end{align}
    holds, where  $\constrTight_i= \norm{\Ptube^{-\sfrac{1}{2}} \begin{bmatrix}I&\K^\top \end{bmatrix}\constrMat_i^\top}$ for $i=1,\dots, \dimconstrVec$.
\end{lemma}
\begin{proof}
    Since Ass.~\ref{ass:iqc} holds, the error bound~\eqref{eq:error_bound_ineq_mpc} from Thm.~\ref{thm:error_bound} is valid.
    This implies 
    \begin{align*}
    \constrMat_i \begin{bmatrix}
    \x_{\k|\t} \\ \u_{\k|\t}
    \end{bmatrix} &= \constrMat_i \begin{bmatrix}
    \nomx_{\k|\t} + \e_{\k|\t}\\ \nomu_{\k|\t}+\K \e_{\k|\t}
    \end{bmatrix}\leq \max_{\norm{\bar \e}^2_\Ptube \leq \s_{\k|\t}} \constrMat_i \begin{bmatrix}
    \nomx_{\k|\t} +\bar \e\\ \nomu_{\k|\t}+\K\bar \e
    \end{bmatrix}
    \end{align*}
    where we can solve the maximization problem with the transformation $\tilde \e = \Ptube^{\sfrac 1 2} \bar \e$
    \begin{align*}
    \constrMat_i \begin{bmatrix}
    \x_{\k|\t} \\ \u_{\k|\t}
    \end{bmatrix}&\leq\max_{\norm{\tilde \e}^2 \leq \s_{\k|\t}} \constrMat_i \begin{bmatrix}
    \nomx_{\k|\t} +\Ptube^{-\sfrac 1 2}\tilde\e\\ \nomu_{\k|\t}+\K\Ptube^{-\sfrac 1 2}\tilde \e
    \end{bmatrix}\\
    &=\constrMat_i \begin{bmatrix}
    \nomx_{\k|\t}\\ \nomu_{\k|\t}
    \end{bmatrix} + \norm{\constrMat_i \begin{bmatrix}I\\\K\end{bmatrix}\Ptube^{-\sfrac{1}{2}}} \sqrt{\s_{\k|\t}}.
    \end{align*}
    Hence, we have shown $\constrMat \begin{bmatrix}
    \x_{\k|\t} \\ \u_{\k|\t}
    \end{bmatrix}\leq\constrMat \begin{bmatrix}
    \nomx_{\k|\t} \\ \nomu_{\k|\t}
    \end{bmatrix} + \constrTight \sqrt{\s_{\k|\t}}\leq \constrVec$.
\end{proof}

At each time step $\t\geq 0$ we measure the current state $\x_\t=\x_{0|\t}=\x_{1|\t-1}$ and solve the following optimization problem based on this measurement and the previously predicted nominal state $\nomx_{1|\t-1}$ and tube size $\s_{1|\t-1}$
\begin{subequations}\label{eq:mpc}
\begin{align}\label{eq:mpc_cost}
\min_{\nomu_{\cdot|\t},\nomx_{0|\t}}\,&\sum_{\k=0}^{\T-1}\left(\norm{\nomx_{\k|\t}}_\Q^2+\norm{\nomu_{\k|\t}}_\Rcost^2\right) + \norm{\nomx_{\T|\t}}^2_\S\\
\text{s.t.\ \,}& \text{initial constraint \eqref{eq:s_cl}} \nonumber\\
& \text{nominal dynamics } \eqref{eq:nom}\text{ for } \k =0,\dots,\T-1\nonumber\\
& \text{tube dynamics } \eqref{eq:s_ol}\text{ for } \k =0,\dots,\T-1 \nonumber \\
& \text{tightened constraints } \eqref{eq:mpc_constr}\text{ for } \k =0,\dots,\T-1\nonumber\\ \label{eq:mpc_term}
&\text{terminal constraint }\begin{bmatrix}\nomx_{\T|\t}^\top  & \s_{\T|\t} \end{bmatrix}^\top \in \Omega
\end{align}
\end{subequations}
where $\Omega\subseteq \R^{\dimx+1}$ is the terminal constraint set and $\Q, \S\in\R^{\dimx\times\dimx}$, and $\Rcost\in\R^{\dimu\times\dimu}$ are positive definite weighting matrices. 
Hence, the initial nominal state $\nomx_{0|\t}$ is a decision variable.
We denote the minimizer of problem~\eqref{eq:mpc} by $\nomu_{\cdot|\t}^\star$ and $\nomx_{0|\t}^\star$, with the corresponding trajectories $\nomx_{\cdot|\t}^\star$ and $\s_{\cdot|\t}^\star$.
Then, as is standard in tube-based MPC and as defined in \eqref{eq:input}, the control input is given by $\u_\t=\u_{0|\t}^\star=\K(\x_\t-\nomx_{0|\t}^\star)+\nomu_{0|\t}^\star$.
Note that this implies for the closed loop system that $\x_\t=\x_{0|\t}^\star$, $\filtx_\t=\filtx_{0|\t}^\star$, $\y_\t=\y_{0|\t}^\star$, $\filty_\t=\filty_{0|\t}^\star$, and $\w_\t=\w_{0|\t}^\star$, where the stars denote that these signals result from $\nomu_{\cdot|\t}=\nomu_{\cdot|\t}^\star$ and $\nomx_{0|\t}=\nomx_{0|\t}^\star$. 

Note that the stage cost only acts on the nominal state and input, comparable to standard tube-based MPC designs (e.g., \cite{Chisci2001}, \cite{Mayne2005}, \cite{Rakovic2012}, \cite{Loehning2014}, etc.). Hence, the asymptotic behavior is dominated by $\u=\K\x$. Even further, and similar to \cite{Chisci2001}, \cite{Rakovic2012}, and \cite{Loehning2014}, the MPC optimizer chooses $\nomx_{0|\t}=0$ and $\nomu_{\cdot|\t}=0$ if the controller $\K$ is guaranteed to steer the true system towards $0$ without violating the constraints.
In this sense, the MPC is restrained and interferes only if robust constraint satisfaction of the system controlled by $\K$ cannot be guaranteed.
To ensure that problem~\eqref{eq:mpc} is recursively feasible, we need to design suitable terminal conditions.
\begin{assumption}[Terminal conditions]\label{ass:term_cond}
    The matrices $\Q$, $\Rcost$, and $\S$ are positive definite. The terminal set $\Omega$ contains the origin in its interior and there exist $\Kloc\in\R^{\dimx\times\dimu}$ such that for all $\begin{bmatrix}\nomx^\top & \s \end{bmatrix}^\top\in\Omega$ we have
    \begin{enumerate}
        \item[1)] positive invariance\vspace{-0.2cm}
    \end{enumerate}
        \begin{align*}
        \begin{bmatrix}(\A+\Bu\Kloc)\nomx \\ \rho^2 \s + \|(\C+\Du\Kloc)\nomx\|^2_\Gamma+\gamma\dmax^2 \end{bmatrix} \in \Omega
        \end{align*}\vspace{-0.3cm}
        \begin{enumerate}
        \item[2)] constraint satisfaction: $\displaystyle
        \constrMat\begin{bmatrix}
        \nomx \\ \Kloc\nomx
        \end{bmatrix} \leq \constrVec-\sqrt\s\constrTight$
        \item[3)] terminal cost decrease\vspace{-0.2cm}
        \begin{align*}
        \norm{(\A+\Bu\Kloc)\nomx}_\S^2 - \norm{\nomx}_\S^2 \leq - \norm{\nomx}_{\Q}^2-\norm{\Kloc\nomx}_\Rcost^2.
        \end{align*}
    \end{enumerate}
\end{assumption}
In Sec.~\ref{sec:term_cond} below, we will discuss how to construct $\Kloc$, $\S$ and $\Omega$ that satisfy Ass.~\ref{ass:term_cond}.
Now, we have all ingredients to show that the MPC controller indeed stabilizes the system and guarantees robust constraint satisfaction.
\begin{theorem}[Stability and Recursive Feasibility]\label{thm:stab}
    Let Ass.~\ref{ass:iqc} and~\ref{ass:term_cond} hold with $\rho < 1$. 
    Let $\filtx_0=0$ and assume that the optimization problem~\eqref{eq:mpc} is feasible at time $\t=0$.
    Then~\eqref{eq:mpc} is feasible for all $\t\geq 0$, the constraints~\eqref{eq:constr} are satisfied for all times $\t\geq 0$ and the closed loop satisfies the following ISS\footnote{Eq.~\eqref{eq:iss} is an integrated form of the classical notion for ISS, which was shown to be equivalent in~\cite{Sontag1998} for continuous-time, but the proof follows similar steps in discrete-time.} bound: there are a constant $\constd>0$ and a class~$\mathcal K$ function\footnote{$\alpha:\R_{\geq 0} \to \R_{\geq 0}$ continuous, monotonically increasing and $\alpha (0)=0$.} $\alpha$ such that for all $N\geq 0$ it holds
    \begin{align}\label{eq:iss}
    \sum_{\t=0}^{N} \norm{\x_\t}^2
    &\leq \alpha (\norm{x_0}) + \constd \sum_{\t=0}^{N-1}\norm{\d_\t}^2. 
    \end{align}
\end{theorem}
\begin{proof}
    The proof is divided into three parts.
    
    \emph{Recursive feasibility.}
    We show recursive feasibility by induction. Therefore, assume~\eqref{eq:mpc} is feasible at time $\t-1$, then define the following candidate solution
    \begin{align*}
    \bar \nomu_{\k|\t} &= \begin{cases} \nomusol{\k+1}{\t-1} & \text{for } \k=0,...,\T-2 \\ \Kloc \nomxsol{\T}{\t-1} & \text{for } \k=\T-1
    \end{cases}\\
    \nomxpred{\k}{\t} &= \begin{cases} \nomxsol{\k+1}{\t-1} & \text{for } \k=0,...,\T-1 \\ (\A+\Bu\Kloc) \nomxsol{\T}{\t-1} & \text{for } \k=\T
    \end{cases}\\
    \cpred{\k}{\t} &= \begin{cases} \csol{\k+1}{\t-1} &\hspace{-1cm} \text{for } \k= 0,...,\T-1 \\ \rho^2 \csol{\T}{\t-1}+\|\rpred{\T-1}{\t}\|^2_\Gamma+\gamma\dmax^2 & \text{for } \k=\T,
    \end{cases}
    \end{align*}
    where $\rpred{\T-1}{\t}=(\C+\Du\Kloc)\nomxsol{\T}{\t-1}$. 
    This candidate solution follows the nominal dynamics~\eqref{eq:nom} and the tube dynamics~\eqref{eq:error_bound_mpc}.
    The induction hypothesis yields~\eqref{eq:mpc_constr} for $\k\in [0,\T-2]$ as well as $\begin{bmatrix}\nomx_{\T|\t-1}^{\star \top} & \s_{\T|\t-1}^{\star}\end{bmatrix}^\top\in\Omega$.
    With 2) in Ass.~\ref{ass:term_cond} it follows~\eqref{eq:mpc_constr} for $\k=\T-1$ and with~1) in Ass.~\ref{ass:term_cond} it follows~\eqref{eq:mpc_term}.
    Hence, this candidate solution is feasible.
    
    \emph{Robust constraint satisfaction.} Follows immediately from Lemma~\ref{lem:constr_satisf}, recursive feasibility, and $\x_{\t}=\x_{0|\t}^\star$, $\u_\t = \u_{0|\t}^\star$.
    
    \emph{Input-to-state stability.}
    To derive the ISS bound~\eqref{eq:iss}, we note that the LMI \eqref{eq:lmi} from Ass.~\ref{ass:iqc} implies
    \begin{align*}
    \newcommand{\colspace}{\!\!\!}\colspace
    \begin{bmatrix}
    \\\\ \symLMI \\\\\\
    \end{bmatrix}^\top\!\! \begin{bmatrix}
    -\rho^2\P & \colspace 0 & \colspace 0 & \colspace 0 \\
    0 & \colspace\P & \colspace 0 & \colspace 0\\
    0 & \colspace 0 & \colspace\M & \colspace 0 \\ 0& \colspace 0&\colspace 0&\colspace\!-\bar \gamma I
    \end{bmatrix}\begin{bmatrix}
    I & \colspace 0 & 0\\
    \allA & \colspace\allB & \colspace\begin{bmatrix} \Bd & \colspace \Bu \\ \filtBin\Dd &\colspace \filtBin\Du \end{bmatrix}\\ \allC &\colspace \allD & \colspace\begin{bmatrix}    \filtDin\Dd & \colspace\filtDin\Du \end{bmatrix} \\ 0&\colspace 0&I
    \end{bmatrix}\prec 0
    \end{align*} 
    for some $\bar \gamma>0$.
    This can be seen by applying Finsler's Lemma analogous to the proof of Lemma~\ref{lem:error_bound}.
    Let us recap that $\u_\t=\u_{0|\t}^\star$ and thus $\x_\t=\x_{0|\t}^\star$, $\filtx_\t=\filtx_{0|\t}^\star$,  $\y_\t=\y_{0|\t}^\star$,  $\filty_\t=\filty_{0|\t}^\star$,  $\w_\t=\w_{0|\t}^\star$, and similarly denote the closed-loop nominal input $\nomu_\t=\nomu_{0|\t}^\star$ and state $\nomx_\t=\nomx_{0|\t}^\star$.
    Then, by multiplying $\begin{bmatrix}
    \x_{\tau}^\top & \filtx_{\tau}^\top & \w_{\tau}^\top & \d_{\tau}^\top & (\nomu_{\tau}-\K\nomx_{\tau})^\top
    \end{bmatrix}$ from left to the above LMI and its transpose from right, we obtain with
    \newcommand{\colspace}{\!\!\!}
    \begin{align*}
    &\begin{bmatrix}
    \allA & \allB & \colspace\begin{bmatrix} \Bd & \colspace \Bu \\ \filtBin\Dd &\colspace \filtBin\Du \end{bmatrix}\\ \allC &\colspace \allD & \colspace\begin{bmatrix}    \filtDin\Dd & \colspace\filtDin\Du \end{bmatrix} 
    \end{bmatrix}
    \begin{bmatrix}\x_{\tau} \\ \filtx_{\tau} \\ \w_{\tau} \\ \d_{\tau} \\ \nomu_{\tau}-\K\nomx_{\tau}    \end{bmatrix} 
    \\
    &\quad\refeq{(\ref{eq:y}, \ref{eq:input},  \ref{eq:aug_dyn})}{=}\quad \begin{bmatrix}
    \A\x_{\tau}+\Bw\w_{\tau} +\Bd\d_{\tau}+\Bu\u_{\tau} \\ 
    \filtA \filtx_{\tau} + \filtBin \y_{\tau} + \filtBout  \w_{\tau} \\ 
    \filtC \filtx_{\tau} + \filtDin \y_{\tau} + \filtDout  \w_{\tau} 
    \end{bmatrix}\ \ \refeq{(\ref{eq:sys_x}, \ref{eq:filt_ss})}{=}\ \     \begin{bmatrix}\x_{\tau+1} \\ \filtx_{\tau+1} \\ \filty_{\tau}
    \end{bmatrix}
    \end{align*}
     the dissipation inequality
    \begin{align*}
    \begin{split}
    &\norm{\begin{bmatrix}
        \x_{\tau+1} \\
        \filtx_{\tau+1}
        \end{bmatrix}}_\P^2 - \rho^2 
    \norm{\begin{bmatrix}
        \x_{\tau} \\
        \filtx_{\tau}
        \end{bmatrix}}_\P^2 
    \\ & \qquad + \filty^\top_{\tau} \M \filty_{\tau}-\bar \gamma \norm{ \d_{\tau}}^2- \bar \gamma\norm{\nomu_{\tau}-\K\nomx_{\tau}}^2\leq 0.
    \end{split}
    \end{align*}
    Multiplying this inequality with $\rho^{2(\t-\tau-1)}$, summing it up from $\tau =0$ to $\tau=\t-1$, and using \eqref{eq:hardIQC} yields
    \begin{align*}
    \norm{\x_{\t}}_\Ptube^2\refeq{\eqref{eq:Ptube}}{\leq} &\,\norm{\begin{bmatrix}
        \x_{\t} \\
        \filtx_{\t}
        \end{bmatrix}}_\P^2\leq \rho^{2\t}\norm{
        \x_{0}}_{\P_{11}}^2\\&+ \sum_{\tau=0}^{\t-1}\rho^{2(\t-\tau-1)}\bar \gamma(\norm{ \d_{\tau}}^2+ \norm{\nomu_{\tau}-\K\nomx_{\tau}}^2).
    \end{align*}
    Summing this inequality once more from $\t=0$ to $\t=N$ and using the geometric series $\sum_{\t=0}^{N} \rho^{2\t} \leq \frac 1 {1-\rho^2}$ yields
    \begin{align}\nonumber
    \sum_{\t=0}^N\norm{\x_{\t}}_\Ptube^2\leq &\frac{1}{1-\rho^2 }\norm{
        \x_{0}}_{\P_{11}}^2\\&+  \sum_{\t=0}^{N-1}\frac{\bar \gamma}{1-\rho^2 }(\norm{ \d_{\t}}^2+ \norm{\nomu_{\t}-\K\nomx_{\t}}^2).\label{eq:iss1}
    \end{align}
    To proceed, we need to bound the sum over $\nomx_\t$ and $\nomu_\t$. 
    Therefore, let us introduce the notation $\J_\T(\nomx_{\cdot|\t},\nomu_{\cdot|\t})$ for the objective function in problem~\eqref{eq:mpc} and use the suboptimality of the candidate solution to obtain
    \begin{align*}
    &\J_\T(\nomx^\star_{\cdot|\t},\nomu^\star_{\cdot|\t})-\J_\T(\nomx^\star_{\cdot|\t-1},\nomu^\star_{\cdot|\t-1})\\
    &\leq  \J_\T(\nomxpred{\cdot}{\t},\nomupred{\cdot}{\t})-\J_\T(\nomx^\star_{\cdot|\t-1},\nomu^\star_{\cdot|\t-1})\\&= \|\nomxsol{\T}{\t-1}\|_\Q^2 + \|\Kloc\nomxsol{\T}{\t-1}\|_\Rcost - \|\nomxsol{0}{\t-1}\|_\Q^2 - \|\nomusol{0}{\t-1}\|_\Rcost^2 \\ &\quad + \|(\A+\Bu\Kloc)\nomxsol{\T}{\t-1}\|_\S^2 -\|\nomxsol{\T}{\t-1}\|_\S^2,
  \end{align*}
   and further with the third property of Ass.~\ref{ass:term_cond} it follows
   \begin{align*}
    &\J_\T(\nomx^\star_{\cdot|\t},\nomu^\star_{\cdot|\t})-\J_\T(\nomx^\star_{\cdot|\t-1},\nomu^\star_{\cdot|\t-1})
    \\ &\qquad \refeq{\text{Ass.~\ref{ass:term_cond}}}{\leq}- \|\nomxsol{0}{\t-1}\|_\Q^2 - \|\nomusol{0}{\t-1}\|_\Rcost^2 = - \|\nomx_{\t-1}\|_\Q^2 - \|\nomu_{\t-1}\|_\Rcost^2.
    \end{align*}
    If we sum this inequality from $\t=1$ to $\t=N$ we find that 
    \begin{align} \sum_{\t=0}^{N-1}\left(\big\|\nomx_{\t}\big\|_\Q^2+\big\|\nomu_{\t}\big\|_\Rcost^2\right)&\leq\J_\T(\nomx^\star_{\cdot|0},\nomu^\star_{\cdot|0})- \J_\T(\nomx^\star_{\cdot|N},\nomu^\star_{\cdot|N})\nonumber\\
        &\leq\J_\T(\nomx^\star_{\cdot|0},\nomu^\star_{\cdot|0})\leq \alpha_0 (\norm{\x_0}),\label{eq:iss2}
    \end{align}
    where the second inequality holds due to non-negativity of $\J_\T(\nomx^\star_{\cdot|N},\nomu^\star_{\cdot|N})$ for all $N$ and the third inequality is discussed in the following. 
    For $(\x_{0},0)\in\Omega$ we know that problem \eqref{eq:mpc} is feasible, since we can choose $\bar \nomx_{0|0}=\x_0$, which implies $\bar \s_{0|0}=\|\bar \e_{0|0}\|_{\P_{11}}^2=0$, and the local controller $\bar \nomu_{\k|0}=\Kloc \bar \nomx_{\k|0}$, which is feasible due to Ass.~\ref{ass:term_cond}.
    Thus, we conclude $\J_\T(\nomx^\star_{\cdot|0},\nomu^\star_{\cdot|0})\leq\J_\T(\bar \nomx_{\cdot|0},\bar \nomu_{\cdot|0})\leq \norm{\bar \nomx_{0|0}}_\S^2=\norm{\x_0}_\S^2$, where the second inequality follows from repeatedly applying condition 3) of Ass.~\ref{ass:term_cond}.
    Since the origin is in the interior of $\Omega$ this bound holds in a neighborhood of $\x_0=0$.
    We can extend such a bound by a class $\mathcal K$ function $\alpha_0$ over the whole feasible set, i.e., $\J_\T(\nomx^\star_{\cdot|0},\nomu^\star_{\cdot|0})\leq \alpha_0(\norm{\x_0})$ (see \cite[Prop~B.25]{Rawlings2017}) due to local boundedness of $\J_\T(\nomx^\star_{\cdot|0},\nomu^\star_{\cdot|0})$ for feasible $\x_0$.
    Hence, it follows~\eqref{eq:iss2}.
    The positive definiteness of $\Q$ and $\Rcost$ ensures existence of a constant $\constproof_2>0$ such that
    \begin{align}\label{eq:iss3}
    \constproof_2 \sum_{\t=0}^{N-1} \! \norm{\nomu_{\t}-\K\nomx_{\t}}^2 \leq\!\! \sum_{\t=0}^{N-1}\left(\big\|\nomx_{\t}\big\|_\Q^2+\big\|\nomu_{\t}\big\|_\Rcost^2\right)\refeq{\eqref{eq:iss2}}{\leq}  \alpha_0 (\norm{\x_0})
    \end{align} 
    for all $N$.
    If we use~\eqref{eq:iss3} in~\eqref{eq:iss1} and do some basic algebra to estimate the positive definite weighting matrices, then we obtain that there exist constants $\constproof_3,\constproof_4,\constproof_\d>0$ such that
    \begin{align*}
    \sum_{\t=0}^N\norm{\x_{\t}}^2\ \ \refeq{(\ref{eq:iss1},\,\ref{eq:iss3})}{\leq}\  \,\constproof_4 \alpha_0 (\norm{\x_0})+\constproof_3\norm{
        \x_{0}}^2+ \constproof_\d \sum_{\t=0}^{N-1}\norm{ \d_{\t}}^2.
    \end{align*}
    Defining $\alpha(\norm{\x_0})=\constproof_3\norm{
        \x_{0}}^2+\constproof_4 \alpha_0 (\norm{\x_0})$ concludes the proof.
\end{proof}

\begin{remark}
    The bound~\eqref{eq:iss} can be formulated with $\alpha(\x_0)=a_0 \|\x_0\|^2$ if only a compact set of initial conditions is considered, e.g., due to compact constraints.
    In this case, the function $\alpha_0$ in the proof can be constructed as quadratic function from the local quadratic bound and the maximum of $\J_\T$ on this compact set.
    The difficulty in achieving a quadratic bound without considering compact sets stems from the nonlinear constraints.
    Such a quadratic bound is desirable, since it guarantees not only asymptotic but exponential stability in the absence of disturbances.
\end{remark}

\subsection{Extensions, special cases and discussion}
In this subsection, we discuss some special cases of the scheme and further extensions.
\begin{remark}\label{rem:simplifications}
    In special cases, the recursion of the error bound $\s_{0|\t}$ in \eqref{eq:s_cl} can be simplified.
    \begin{itemize}
        \item If the initialization of the nominal predictions is set to follow the nominal dynamic, i.e., $\nomx_{0|\t+1}=\nomx_{1|\t}$, then \eqref{eq:s_cl} simplifies to $\s_{0|\t+1}=\s_{1|\t}$.
        This is the special case that has been addressed in the preliminary conference paper~\cite{Schwenkel2020}.
        \item If the LMI \eqref{eq:lmi} from Ass.~\ref{ass:iqc} can be satisfied with a blockdiagonal $\P$ having $\P_{21}=0$, then \eqref{eq:s_cl} simplifies to 
        \begin{align*}
            \s_{0|\t}=\s_{1|\t-1}+\norm{\e_{0|\t}}_{\P_{11}}^2-\norm{\e_{1|\t-1}}_{\P_{11}}^2
        \end{align*}
        \item If the filter state $\filtx_\t=\filtx_{0|\t}=\filtx_{1|\t-1}$ is known, then the tighter recursion \eqref{eq:c_cl} can be used instead of \eqref{eq:s_cl} to propagate $\s_{1|\t-1}$ to  $\s_{0|\t}$. The filter state can be computed if (i) $\filtBout=0$ or $\filtBout\w_\t$ can be measured; and (ii) $\filtBin=0$ or $\filtBin\y_\t$ can be measured. If the filter is static we can also use \eqref{eq:c_cl}.
    \end{itemize}
\end{remark}
\begin{remark}\label{rem:constr_trafo}
    The increase in the computational complexity is moderate compared to a nominal MPC scheme.
    The scalar error bounding system can be interpreted as an additional state such that the number of decision variables increases as if the state dimension would increase by $1$.
    However, we introduced nonlinear constraints~\eqref{eq:error_bound_mpc} and~\eqref{eq:mpc_constr}, which might render the problem more complicated.
    Nevertheless, we can reformulate the non-differentiable square root in the constraint~\eqref{eq:mpc_constr} as an equivalent differentiable constraint
    \begin{align*}
        \eqref{eq:mpc_constr} &\Leftrightarrow \sqrt{\s_{\k|\t}}\constrTight_i \leq \constrVec_i-\constrMat_i \begin{bmatrix}
            \nomx_{\k|\t} \\ \nomu_{\k|\t}
        \end{bmatrix} \ \forall i=1,...,\dimconstrVec
        \\&\Leftrightarrow \s_{\k|\t}\constrTight_i^2\leq \left(\constrVec_i-\constrMat_i \begin{bmatrix}
            \nomx_{\k|\t} \\ \nomu_{\k|\t}
        \end{bmatrix}\right)^2 \ \wedge\ \constrVec_i\geq \constrMat_i \begin{bmatrix}
            \nomx_{\k|\t} \\ \nomu_{\k|\t}
        \end{bmatrix}\\&\qquad \qquad \qquad\qquad \qquad \qquad\qquad \qquad  \ \forall i=1,...,\dimconstrVec.
    \end{align*}
    An analogous transformation can be applied to the constraint~\eqref{eq:s_cl} as well to get rid of the square root therein.
    If we further combine this with one of the simplifications from Rem.~\ref{rem:simplifications}, where~\eqref{eq:s_cl} need not be included, then the optimization problem becomes a Quadratically Constrained Quadratic Program (QCQP). 
\end{remark}
\begin{remark}
    We note that the proposed MPC scheme can be further simplified by using a fixed tube size $\s_{\text{max}}$ instead of the tube dynamics.
    Then, a constraint on the nominal output $\|\nomy_{\k|\t}\|_\Gamma^2 \leq (1-\rho^2) \s_{\text{max}} - \gamma \dmax^2$ can be used to make sure that the actual $\s_{\k|\t}$ is always less than or equal to $\s_{\text{max}}$.
    This special case of using a constant tube in combination with an output constraint is conceptually similar to~\cite{Falugi2014}, where exactly this procedure is proposed with an $\ell_\infty$-gain bound on $\Delta$ instead of an IQC describing it.
    If we further want to optimize over the initial nominal state, we need to add a constraint on $\nomx_{0|\t}$ that~\eqref{eq:s_cl} is less than or equal to $\s_{\text{max}}$ if we substitute $\s_{1|\t-1}=\s_{\text{max}}$ in~\eqref{eq:s_cl}.
\end{remark}

    \begin{remark}\label{rem:optimize_P}
        The design parameters for the tube dynamics and the constraint tightening are $\rho$, $\P$, $\Gamma$, $\gamma$, $\K$.
        When the pre-stabilizing control law $\K$ and the exponential decay rate $\rho$ are fixed, the parameters $\P$, $\Gamma$ and $\gamma$ can be determined as solutions of the LMI \eqref{eq:finsler_lmi}.
        Ass.~\ref{ass:iqc} guarantees the existence of not only one but infinitely many solutions of the LMI (due to the strict inequality), which we can use to optimize over the matrices $\P$ that define the shape of the tube.
        Finding the matrix $\P$ that minimizes the constraint tightening $\constrTight$ is actually a convex problem since $\constrTight_i^2 \leq \gamma_{i}$ with $\gamma_i>0$ can be reformulated as an LMI by using $\Ptube=\P_{11}-\P_{21}^\top \P_{22}^{-1}\P_{21}$ and applying the Schur complement twice
        \begin{align*}
            \constrTight_i^2 \leq \gamma_{i} &\Leftrightarrow \constrMat_i \begin{bmatrix}
                I& \K
            \end{bmatrix}^\top \Ptube^{-1} \begin{bmatrix}
                I & K
            \end{bmatrix}\constrMat_i^\top \leq \gamma_{i}\\
            &\Leftrightarrow \begin{bmatrix}
                \Ptube & \begin{bmatrix}
                    I & \K 
                \end{bmatrix} \constrMat_i^\top \\ \constrMat_i \begin{bmatrix}
                    I & \K
                \end{bmatrix}^\top & \gamma_{i}
            \end{bmatrix} \succeq 0\\ 
            &\Leftrightarrow \begin{bmatrix}
                \P_{22} & \P_{21} & 0  \\
                \P_{21}^\top & \P_{11} & \begin{bmatrix}
                    I & \K 
                \end{bmatrix} \constrMat_i^\top \\
                0 & \constrMat_i \begin{bmatrix}I & \K\end{bmatrix}^\top & \gamma_{i}
            \end{bmatrix} \succeq 0.
        \end{align*}
        Choosing $\sum_{i=0}^{\dimconstrVec} \gamma_{i}$ as objective yields a semi-definite program whose solution is a matrix $\P$ that minimizes the sum of all constraint tightenings.
        Further, since we can always rescale a solution of the LMI, we need to fix or at least bound $\Gamma$ and $\gamma$ when performing this optimization, otherwise the solutions of $\P$, $\Gamma$ and $\gamma$ tend to infinity.
    \end{remark}

\begin{remark}
    The shape of the tube resulting from the proposed approach is an ellipsoid defined by the shape matrix $\Ptube=\P_{11}-\P_{21}^\top \P_{22}^{-1} \P_{21}$.
    Using a fixed shape for the tube is important to keep the online computational complexity low and is standard in most tube-based MPC schemes (e.g., \cite{Chisci2001}, \cite{Mayne2005}, \cite{Rakovic2012}, \cite{Loehning2014}, etc.).
    Nevertheless, we can reduce conservatism by using several tubes with different shape matrices $\Ptube_i$ and scaling parameters $\s_{\k|\t}^i$ at the same time leading to an intersection of ellipsoids and a vector valued tube scaling parameter $\s_{\k|\t}$.
    Note that the feedback $\K$ must be the same for all tubes.
    Then, we can use Rem.~\ref{rem:optimize_P} but choose only one $\gamma_i$ as objective to obtain the shape matrix $\Ptube^i$ which constitutes the tube to tighten constraint $i$ (and only constraint $i$). 
    By repeating this for all constraints, we obtain $\dimconstrVec$ ellipsoidal tubes and use the intersection of them for the constraint tightening.
    If we do not optimize over initial conditions (compare Rem.~\ref{rem:simplifications}) and use the same $\Gamma$ and $\gamma$, then \eqref{eq:s_cl} is independent of $\Ptube^i$, which implies that all tube scalings $\s_{\k|\t}^i$ are identical, thus we need only one $\s_{\k|\t}$ and in this case do not increase the computational complexity.
\end{remark}

\subsection{Terminal ingredients}\label{sec:term_cond}
The purpose of this section is to give a constructive proof how a local controller $\K_\Omega$, a terminal set $\Omega$ and a terminal cost weight $\S$ can be found that satisfy Ass.~\ref{ass:term_cond}.
\begin{proposition}\label{prop:term_cond}
    Let the matrices $\Q\succ 0$, $\Rcost\succ 0$, $\A$, $\Bu$, $\C$, $\Du$, $\constrMat$, the vectors $\constrVec$, $\constrTight$ and the scalars $\rho \in (0,1), \dmax\geq 0,\gamma>0$ be given. If $(\A,\Bu)$ is stabilizable and $\constrVec> \frac{\sqrt{\gamma}\dmax}{\sqrt{1-\rho^2}} \constrTight$, then there exists $\Kloc$, $\S\succ0$, $\nomx_\Omega>0$, and $\s_\Omega>0$ such that Ass.~\ref{ass:term_cond} holds with $\Omega=\big\{\begin{bmatrix}\nomx^\top & \s\end{bmatrix}^\top\big|\|\nomx\|_\S^2 \leq \nomx_\Omega, 0\leq \s\leq \s_\Omega\big\}$.
\end{proposition}
\begin{proof}
    Since $(\A,\Bu)$ is stabilizable, we can find $\Kloc$ such that $\A+\Bu\Kloc$ is Schur stable. Thus, for each $\Qproof\succ0$ there is $\S\succ0$ such that
    \begin{align*}
    (\A+\Bu\Kloc)^\top \S (\A+\Bu\Kloc) - \S =-\Qproof.
    \end{align*}
    Setting $\Qproof=\Q+\Kloc^\top \Rcost \Kloc$ renders 3) of Ass.~\ref{ass:term_cond} true for all $\nomx$.
    Further, for each $\nomx_\Omega>0$, the set $\Omega_{\nomx}=\{\nomx|\nomx^\top \S \nomx\leq \nomx_\Omega \}$ is a positive invariant set of the nominal dynamics~\eqref{eq:nomx} controlled by $\nomu=\Kloc \nomx$. To choose $\s_\Omega$ such that $\Omega$ is a positive invariant set of the augmented dynamics of $\begin{bmatrix}\nomx^\top &\s \end{bmatrix}^\top$ and hence 1) of Ass.~\ref{ass:term_cond} holds, we set
    \begin{align*}
    \s_\Omega :=\sup_{\nomx \in \Omega_{\nomx}} \frac{1}{1-\rho^2} \big(\|(\C+\Du\Kloc)\nomx\|_\Gamma^2 + \gamma \dmax^2 \big).
    \end{align*}
    Finally, we can choose $\nomx_\Omega>0$ and $\s_\Omega >\frac{\gamma \dmax^2}{1-\rho^2}$ small enough such that $
    \constrMat\begin{bmatrix}
    \nomx \\ \Kloc\nomx
    \end{bmatrix} \leq \constrVec-\sqrt{\s_\Omega}\constrTight$
    holds for all $\|\nomx\|_\S^2\le\nomx_\Omega$, since $\constrVec> \frac{\sqrt{\gamma}\dmax}{\sqrt{1-\rho^2}} \constrTight$. Then 2) of Ass.~\ref{ass:term_cond} holds as well.
\end{proof}
\begin{remark}
    If the requirement $\constrVec> \frac{\sqrt{\gamma}\dmax}{\sqrt{1-\rho^2}} \constrTight$ is not satisfied, then the worst-case disturbance $\dmax$ is too large to meet the constraints, such that no suitable terminal region exists (for this choice of $\P$ and $\K$ in Ass.~\ref{ass:iqc}, other $\P$ and $\K$ could change $\constrTight$.).
    Such a requirement is intuitive as constraint satisfaction cannot be achieved if the disturbances get arbitrarily large.    
\end{remark}

\begin{remark}\label{rem:pseudoAlgorithm}
  We want to briefly summarize the main steps and offline computations necessary to implement the scheme.
  \begin{enumerate}
    \item Find a $\rho$-hard IQC description of the uncertainty $\Delta$.
    \item Find suitable $\K$ such that Ass.~\ref{ass:iqc} holds.
    \item Compute $\P$, $\Gamma$ and $\gamma$ according to Rem.~\ref{rem:optimize_P}.
    \item Compute terminal ingredients according to Prop.~\ref{prop:term_cond}.
  \end{enumerate}
  A systematic synthesis procedure of step 2) is subject of current research. 
  In contrast to step 3), it cannot be expected to result in a semi-definite program (compare \cite{Veenman2014}, where the IQC synthesis is solved iteratively similar to a $D$-$K$-iteration).
\end{remark}

\section{Numerical Example}\label{sec:example}
The following example demonstrates the advantages of using the much more flexible IQC framework to describe dynamic uncertainties compared to the $\ell_\infty$-gain that was used in earlier tube-based MPC schemes \cite{Falugi2014}.
To this end, consider the following system
\begin{align*}
\x_{\t+1} &= \begin{bmatrix} 1.05 & -0.3 \\ 0 & 0.95 \end{bmatrix} \x_\t +\begin{bmatrix}
1 \\ 0
\end{bmatrix} \d_\t + \begin{bmatrix}
0 \\ 1
\end{bmatrix} \u_{\t-\tau_\t}\\& = A\x_\t + \Bd \d_\t + \Bu \u_{\t-\tau_\t} 
\end{align*}
with an unknown, possibly time-varying delay $\tau_\t \in [0, \tau_\mathrm{max}]$, $\tau_\mathrm{max}=2$ on the input signal $u$ and with an external disturbance $\d$ that satisfies $|\d_\t|\leq 0.001$ and acts on the unstable mode.
Note that the control input (even for $\tau_t=0$) has a larger relative degree to the unstable mode than the disturbance and additionally must go through the time delay.
Further, the state constraint $\left[\begin{smallmatrix} -0.4 \\ -0.2 \end{smallmatrix}\right] \leq \x_\t \leq \left[\begin{smallmatrix} 0.4 \\ 0.2 \end{smallmatrix}\right]$ and the input constraint $|\u_\t|\leq 0.1$ must be satisfied at all times.
In order to write the system in the form of \eqref{eq:sys}, we define the nominal case as $\tau = 0$.
Hence, we obtain
\begin{subequations}
	\begin{align}
	\x_{\t+1} &=\A\x_\t+ \Bu \w_\t +\Bd \d_\t +\Bu \u_\t \\
	\y_\t &= \u_\t \\
	\w_\t &= \Delta (\y )_\t = \y_{\t-\tau_\t} - \y_\t. \label{eq:Delta_ex} 
	\end{align}	
\end{subequations}
It is straightforward to see that $\Delta$ is a causal bounded operator with $\ell_\infty$-gain of $2$.
However, based on the only information of the $\ell_\infty$-gain of $\Delta$, the unstable system cannot be robustly stabilized as the $\ell_\infty$-gain bound of $2$ includes the case $\Delta(\y)_\t=-\y_\t$ which cancels all inputs.
Thus, the approach from \cite{Falugi2014} cannot be used for this problem and we need a less conservative description of the uncertainty $\Delta$ as for example via IQCs.
As proposed in this article, we can design a tube-based MPC scheme based on IQCs.
Hence, we first define the filter $\filt= \ssrep{\filtA}{\begin{bmatrix}\filtBin&\filtBout\end{bmatrix}}{\filtC}{\begin{bmatrix}\filtDin&\filtDout\end{bmatrix}}$ with 
\begin{align*}
    \filtA &= \left[\begin{smallmatrix}0 & I & &  \\[-0.15cm] & \ddots & \ddots &  \\[-0.15cm] && \ddots & I \\ &&& 0\end{smallmatrix}\right],& \filtBin &=\left[\begin{smallmatrix}0 \\[-0.15cm] \vdots \\0 \\ I \end{smallmatrix}\right], & \filtBout &= 0 \\
 \filtC&=\left[\begin{smallmatrix}I & -I & &  \\[-0.15cm] & \ddots & \ddots &  \\[-0.15cm] && \ddots & -I \\ &&& I \\ 0 &\dots&\dots& 0\end{smallmatrix}\right],& \filtDin &=\left[\begin{smallmatrix}0 \\[-0.15cm] \vdots \\0 \\ -I\\0\end{smallmatrix}\right], & \filtDout &= \left[\begin{smallmatrix}0 \\[-0.15cm] \vdots \\0 \\ I \end{smallmatrix}\right],
\end{align*}
which results in the filter state $\filtx_{\t} = \begin{bmatrix}
    \y_{\t-\tau_\mathrm{max}}^\top & \dots & \y_{\t-1}^\top 
\end{bmatrix}^\top$ and the output 
\begin{align*}
    \filty_{\t} &= \begin{bmatrix}
        \y_{\t-\tau_\mathrm{max}}^\top - \y_{\t-\tau_\mathrm{max}+1}^\top & \dots& \y_{\t-1}^\top-\y_{\t}^\top & \w_\t^\top
    \end{bmatrix}^\top.
\end{align*}
As next step we show that the delay uncertainty satisfies an IQC described by this filter. Therefore, let $X\in \R^{\dimy\times \dimy}$, $X\succeq 0$ be arbitrary and let us denote the $\tau\times \tau$ all ones (all zeros) matrix by $\mathbbm{1}_\tau$ (by $0_\tau$) and the Kronecker product by~$\otimes$. Further, for $\tau \in [0,\tau_\mathrm{max}]$ let $X_\tau = \diag(0_{\tau_\mathrm{max}-\tau}, \mathbbm{1}_{\tau})\otimes X \in \R^{\tau_\mathrm{max}\dimy \times \tau_\mathrm{max}\dimy}$ and $\M_\tau = \diag(X_\tau, -X)$.
Then, we obtain 
\begin{align*}
0 &=\|\y_{\t-\tau_\t} - \y_\t\|^2_X- \|\w_\t\|^2_X \\ &= \big\|\textstyle\sum_{k=1}^{\tau_\t} (\y_{\t-k} - \y_{\t-k+1})\big\|^2_X - \|\w_\t\|^2_X = \|\filty_\t\|_{M_{\tau_t}}^2.
\end{align*}
Hence, $\Delta$ satisfies the $\rho$-hard IQC defined by $(\filt, M)$ if $M$ satisfies for all $\tau = 0, \dots, \tau_\mathrm{max}$ the LMI $M\succeq M_\tau$ independent of $\rho$.
We choose\footnote{Here, $\rho$ and $\K$ were manually chosen by an LQR design by varying $\rho$ and the LQR weights until Ass.~\ref{ass:iqc} became feasible.} $\rho=0.95$ and $\K=\begin{bmatrix}0.18 & -0.35 \end{bmatrix}$ and observe that the resulting semi-definite program consisting of the LMI constraints~\eqref{eq:finsler_lmi}, $M\succeq M_\tau$, $\Gamma\succ 0$, $\gamma \succ 0$, and $X\succeq 0$, as well as the decision variables $\P$, $\M$, $X$, $\Gamma$, $\gamma$ is solved with the objective described in Rem.~\ref{rem:optimize_P}, which yields
\begin{align*}
\P&\approx \scriptsize \begin{bmatrix}
   5.9 &\!\!  -8.1 &\!\!  -4.1&\!\!  -11.7  \\
  -8.1 &\!\!  15.7 &\!\!   6.0&\!\!  22.2 \\
  -4.2 &\!\!   6.0 &\!\!  40.2&\!\!  -17.0\\
 -11.7 &\!\!  22.2 &\!\! -17.0&\!\!   81.7
\end{bmatrix}, & \M&\approx \scriptsize\begin{bmatrix}
  29.0 &\!\!  14.5 &  0 \\
  14.5 &\!\!  25.4 &  0 \\
  0    &\!\!    0 &\!\!\!\! -20.7 \end{bmatrix}, \\
\gamma &= \Gamma \approx 244, & X &\approx 20.7.
\end{align*}
The cost function~\eqref{eq:mpc_cost} is defined by $Q=I$ and $R=1$ and the prediction horizon $\T=25$.
With the help of Prop.~\ref{prop:term_cond} we find that the terminal ingredients $\Kloc \approx \begin{bmatrix}0.19 & -0.28\end{bmatrix}$, $S\approx\left[\begin{smallmatrix}9.2& -5.6 \\ -5.6&  7.7\end{smallmatrix}\right]$, $\nomx_\Omega\approx 0.0039$, $\s_\Omega=0.1$ satisfy the requirements of Ass.~\ref{ass:term_cond}.
The MPC optimization problems\footnote{To overcome numerical problems with square roots in the constraints, we apply the equivalence transformation discussed in Rem.~\ref{rem:constr_trafo}.} are solved using CasADi~\cite{Andersson2019} with the solver IPOPT. 

The simulation results for two different initial conditions are shown in Fig.~\ref{fig:exmp}.
\begin{figure}
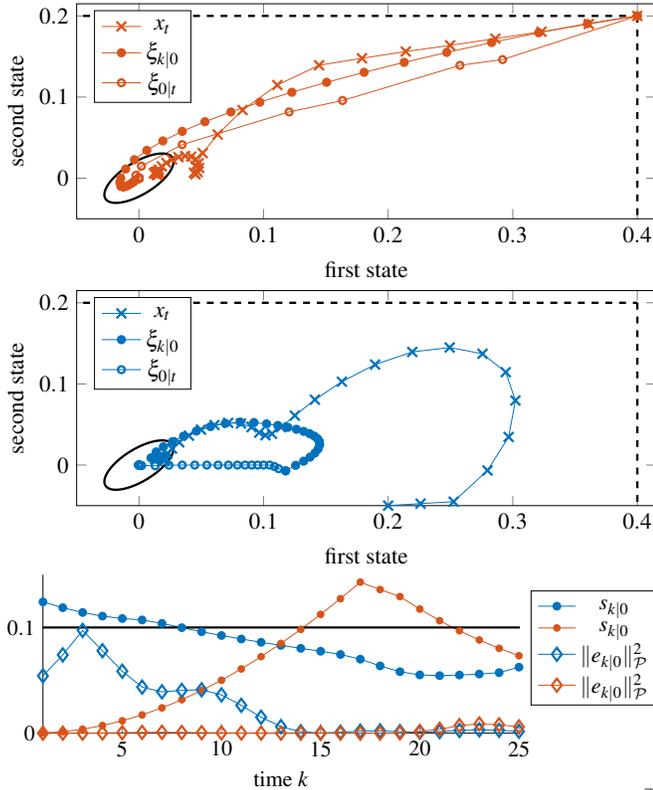

    \centering
    \resizebox{\linewidth}{!}{
        \input{figures/exmp_ss2.tex}}\\
    \resizebox{\linewidth}{!}{
        \input{figures/exmp_ss.tex}}\\
    \resizebox{\linewidth}{!}{
%
%
\definecolor{blau1}{rgb}{0.00000,0.44700,0.74100}%
\definecolor{orange1}{rgb}{0.85000,0.32500,0.09800}%
\begin{tikzpicture}

\begin{axis}[%
width=3in,
height=1in,
at={(0.758in,0.481in)},
scale only axis,
xmin=1,
xmax=25,
ymin=0,
ymax=0.15,
ytick = \empty,
xtick = {0, 5,10, 15,20,25},
axis background/.style={fill=white},
axis x line*=bottom,
axis y line*=left,
xlabel={time $\k$},
ylabel near ticks,
ylabel={},
ytick={0, 0.1},
xlabel near ticks,
legend style = {anchor=north, at={(1.16,0.9)}},
]
\addplot [color=blau1, line width=0.4pt, mark size=1.5pt, mark=*, mark options={solid, blau1, line width=1.0pt}]
  table[row sep=crcr]{%
    0	0.126334475035224\\
    1	0.124211683338295\\
    2	0.118735324080093\\
    3	0.114174145693706\\
    4	0.11075389958458\\
    5	0.108476215690932\\
    6	0.107118817881244\\
    7	0.103867205077468\\
    8	0.0998064898890887\\
    9	0.0959673673580662\\
    10	0.0924021234078202\\
    11	0.0890861895929629\\
    12	0.0859878904273382\\
    13	0.0830697319681144\\
    14	0.0802797158877009\\
    15	0.0775174944057626\\
    16	0.074460164925923\\
    17	0.069998578032789\\
    18	0.0636577609453998\\
    19	0.0581651244429402\\
    20	0.055143211850868\\
    21	0.0544723769422451\\
    22	0.0550242769974622\\
    23	0.0559425058623207\\
    24	0.0572942131518045\\
    25	0.0624602005116584\\
};
\addlegendentry{$\s_{\k|0}$};

\addplot [color=orange1, line width=0.4pt, mark size=1.5pt, mark=*, mark options={solid, orange1}]
  table[row sep=crcr]{%
    0	6.24143424651072e-09\\
    1	0.000320973917760239\\
    2	0.00124139231950626\\
    3	0.00373621856254093\\
    4	0.00724470312468514\\
    5	0.0117263367618731\\
    6	0.0173002104718862\\
    7	0.0239512144768169\\
    8	0.0316566768963943\\
    9	0.040388434202702\\
    10	0.0501143032262291\\
    11	0.0607991727166259\\
    12	0.0724058369528367\\
    13	0.0848956459142019\\
    14	0.0982290201613301\\
    15	0.112365863056475\\
    16	0.127265893295975\\
    17	0.142888913044488\\
    18	0.135872352250887\\
    19	0.12920518768276\\
    20	0.117438040766563\\
    21	0.106439342079242\\
    22	0.0968476086882255\\
    23	0.0881684851358945\\
    24	0.0802677118766688\\
    25	0.0732497833938129\\
};
\addlegendentry{$\s_{\k|0}$};
\addplot [color=blau1, line width=0.4pt, mark size=3pt, mark=diamond, mark options={solid, blau1, line width=1pt}]
table[row sep=crcr]{%
  1	0.0542564194933285\\
  2	0.0740751679995947\\
  3	0.0972233593757962\\
  4	0.0779786807756689\\
  5	0.0586250582593481\\
  6	0.0436019060494408\\
  7	0.0391060047906156\\
  8	0.0404776044017753\\
  9	0.0410247588693748\\
  10	0.036227703795129\\
  11	0.026352936207738\\
  12	0.0151111692698647\\
  13	0.00619451016305353\\
  14	0.00138260622863075\\
  15	0.000200324887746817\\
  16	0.000943392405809566\\
  17	0.00191641880056439\\
  18	0.00211533256689592\\
  19	0.00146185907684446\\
  20	0.00099816885387474\\
  21	0.00151043537675809\\
  22	0.00254031778413223\\
  23	0.0031099147957659\\
  24	0.0027269743355951\\
  25	0.00174247369225442\\
};
\addlegendentry{$\|\e_{k|0}\|_\Ptube^2$};
\addplot [color=orange1, line width=0.4pt, mark size=3pt, mark=diamond, mark options={solid, orange1, line width=1pt}]
table[row sep=crcr]{%
  1	1.18137603427911e-08\\
  2	1.70706761686372e-05\\
  3	1.72373405124673e-05\\
  4	0.000185619079875967\\
  5	0.000476817210990812\\
  6	0.000718292838913378\\
  7	0.000781222842612309\\
  8	0.00068263693227914\\
  9	0.000486365668972163\\
  10	0.000285641836221315\\
  11	0.000135739216657371\\
  12	5.03586179149769e-05\\
  13	1.31949431121696e-05\\
  14	1.48492437724914e-06\\
  15	5.64486418804245e-07\\
  16	6.50545787413325e-06\\
  17	2.06168783843664e-05\\
  18	4.29596483585154e-05\\
  19	0.000272700093879326\\
  20	0.0016181316073794\\
  21	0.00368379562888256\\
  22	0.00720251877685902\\
  23	0.00876004234342835\\
  24	0.00794801728958027\\
  25	0.00594998917709846\\
};
\addlegendentry{$\|\e_{k|0}\|_\Ptube^2$};
\addplot [color=black, line width=1.0pt, forget plot]
table[row sep=crcr]{%
1	0.1\\
31	0.1\\
};
\end{axis}
\end{tikzpicture}%
}\\[-0.25cm]
    \caption{Simulation results for initial conditions $\x_{0}=[0.4\ \ 0.2]^\top$ (orange) and $\x_{0}=[0.2\ \ -0.05]^\top$ (blue). \emph{Top and middle:} state space plots with the state constraints (dashed line) and the terminal region (solid ellipsoid). \emph{Bottom:} tube size at $t=0$, corresponding error, and terminal constraint $\s_{\T|0}\leq \gamma_2$ (solid line).}\label{fig:exmp}
\end{figure}
One of the initial conditions is on the boundary of the constraints and the other one starts close to the eigenspace of the unstable eigenvalue. 
We observe for the first initial condition (top, orange) that the tube size starts very small as the controller must be more cautious when close to the constraints.
When moving away from the constraints, the MPC controller has more freedom and can excite the uncertainty stronger, resulting in a growing tube size until the end of the prediction horizon is approached and the terminal constraint of the tube size $\s$ must be satisfied.
For this initial condition, we can see that the error bound $\|\e_{\k|0}\|_P^2\leq \s_{\k|0}$ is conservative and that the tube grows much faster than the actual error, which indicates some conservatism in the variables $\gamma$ and $\Gamma$ or in the IQC description itself.
In closed loop, we observe that the MPC scheme places the nominal state $\nomx_{0|\t}$ after a few steps directly into the origin.
This behavior can be explained by the fact that the MPC controller is designed to interfere only when necessary. After these few states, the MPC does not need to intervene since the system state is far enough from the constraints and close enough to the origin, such that robust constraint satisfaction is guaranteed when solely applying the pre-stabilizing controller $K$.

For the second initial condition (middle, blue), the controller must be more aggressive in the beginning to push the state from the unstable eigenspace towards the stable one. 
This results in a large tube size in the beginning, which shrinks as the systems state gets closer to the constraints, but grows again at the end of the prediction horizon when larger inputs are needed in order to steer the nominal trajectory into the terminal region.
As we can see, the evolution of the tube size over the prediction horizon is flipped compared to the first initial condition.
Further, we can see that the error bound $\|\e_{\k|0}\|_P^2\leq \s_{\k|0}$ is much tighter for this initial condition, especially at time $\k=2$ the difference between the tube and the actual error is $\approx 15\%$.

Note that when fixing the nominal initial condition $\nomx_{0|\t}$ as in \cite{Schwenkel2020}, the MPC is not initially feasible for the second initial condition $\x_0$ and not even for $0.6\x_0$, which demonstrates the performance increase gained from optimizing $\nomx_{0|\t}$.
Moreover, this example shows how the proposed MPC scheme adjusts the tube size flexibly to different scenarios by optimizing the tube size online and guarantees constraint satisfaction despite the dynamic uncertainty. 
Finally, this example demonstrates that the IQC approach in combination with the scalar tube dynamics reduces conservatism in the sense that it can robustly stabilize a system that cannot be stabilized by describing the uncertainty with an $\ell_\infty$-gain bound and using a static tube in combination with an output constraint as in~\cite{Falugi2014}.

\begin{remark}
    Note that a nominal MPC scheme does not stabilize this example.
    In the neighborhood of the origin where no constraints are active the nominal MPC with the standard LQR terminal cost reduces to an LQR controller.
    However, an LQR with $Q=I$ and $R=1$ for the nominal system does not stabilize the true system.
    Hence, when facing dynamic uncertainties, the robust MPC design is not only needed to handle constraints but also for stability.
    This is in contrast to the case of additive bounded disturbances, where a nominal MPC scheme is already input-to-state stable and a tube-based MPC is only needed to ensure robust constraint satisfaction.
\end{remark}
\section{Conclusion}
We have proposed a tube-based MPC scheme for linear systems subject to dynamic uncertainties and disturbances.
The use of $\rho$-hard IQCs to capture the behavior of the dynamic uncertainty offers a more detailed description than in previous MPC schemes based on $\ell_\infty$-gain bounds.
By extending the $\rho$-hard IQC theory, we were able to derive a dynamic bound on the error between the nominal state and the true system state.
When incorporating this scalar error bounding system to predict the tube size in the MPC scheme, we can ensure recursive feasibility and input-to-state stability.
Finally, we have demonstrated in a numerical example that the proposed scheme can reduce conservatism and is applicable to a larger class of systems compared to existing MPC schemes for dynamic uncertainties. 
An open issue regards the extension to dynamic output feedback and a more detailed investigation of the corresponding offline IQC-based feedback synthesis.

\appendix
\section{Appendix}
\subsection{Proof of Theorem~\ref{thm:pn-hard}}\label{sec:proof_thm_2}
\begin{proof}
    Let use introduce some notation: We conveniently write $\Delta\in \mathrm{IQC} (\rho,\Pi)$ and $\Delta\in\mathrm{hardIQC}(\rho,\filt,\M)$ as short for $\Delta$ satisfies the $\rho$-IQC defined by $\Pi$ and the $\rho$-hard IQC defined by $(\filt,\M)$, respectively. 
    Similarly, we denote the set of matrices $\P=\P^\top$ that satisfy~\eqref{eq:lmi} with $\mathrm{LMI}(\rho,\Psi,\M,\G)$. 
    Further, the operators $\rho_+$ and $\rho_-$ are defined via $(\rho_\pm \circ y)_k = \rho^{\pm k} y_k$ as in~\cite[Definition 3]{Boczar2015}.
        
    In the first part of the proof, we will show that $\rho$-IQCs imply $\rho$-hard IQCs. Note that $\Pi$ is a $\rho$-PN multiplier iff $\Pi_\rho$ is a strict PN multiplier in the sense of~\cite[Definition~4]{Hu2016a}. 
    Hence, we can apply~\cite[Lemma~1 and~6]{Hu2016a} to $\Pi_\rho$ and obtain that there exists a (J-spectral) factorization $(\hat \filt,\hat \M)$, with $\hat \M= \diag(I_\dimy,-I_\dimw)$, $\Pi_\rho = \hat \filt^\sim \hat \M \hat \filt$ and $\hat\filt\in\RHinf$ that has the following properties: 
    (i) $\Delta'$ satisfies the $1$-hard IQC defined by $(\hat \filt,\hat \M)$ for all $\Delta'$ that satisfy the $1$-IQC defined by $\Pi_\rho$, 
    (ii) for any $Y\in\RHinf$: if $\hat \P \in \mathrm{LMI}(1,\hat\filt,\hat\M,Y)$ then $\hat \P\succeq 0$. 
    Defining $\filt = \hat \filt_{\rho^{-1}}$ and $\M=\hat \M$, we see that $\filt_\rho^\sim \M\filt_\rho=\hat \filt^\sim \hat \M \hat \filt=\Pi_\rho$ and $\filt_\rho=\hat\filt \in\RHinf$, i.e., $(\filt,\M)$ is a $\rho$-factorization of $\Pi$.
    Further, we define $\Delta' = \rho_-\circ(\Delta \circ \rho_+)$ and obtain by using \cite[Prop.~7]{Boczar2015} that $\Delta\in \mathrm{IQC}(\rho,\Pi)\Rightarrow \Delta'\in \mathrm{IQC}(1,\Pi_\rho)$.
    Now we can use (i) to conclude $\Delta\in \mathrm{IQC}(\rho,\Pi)\Rightarrow\Delta'\in \mathrm{hardIQC}(1,\hat \filt,\hat \M)$.
    If we take a detailed look at this hard IQC, which holds for all $\y\in\ltwoe{\dimy}$ and thus as well for all $\y':=\rho_-\circ \y$, we observe in two steps: first,
    \begin{align*}
        \filty' = \hat \filt \begin{bmatrix}\y'\\\Delta'(\y')\end{bmatrix}= \hat \filt \circ \rho_- \begin{bmatrix}\y\\\Delta(\y)\end{bmatrix}= \rho_- \circ \filt \begin{bmatrix}\y\\\Delta(\y)\end{bmatrix}
    \end{align*}
    and second, for $p:=\filt \left[\begin{smallmatrix}\y\\\Delta(\y)\end{smallmatrix}\right]=\rho_+ \circ p'$
    \begin{align*}        
        \textstyle\sum_{\t=0}^{\T-1} \rho^{-2\t}\filty_\t^\top \M\filty_\t=\sum_{\t=0}^{\T-1} {\filty'_\t}^{\top} \hat \M\filty_\t'\geq 0.
    \end{align*}
    Thus, we have just shown $\Delta '\in \mathrm{hardIQC}(1,\hat \filt,\hat \M)\Rightarrow\Delta\in \mathrm{hardIQC}(\rho,\filt,\M)$ and altogether $\Delta\in \mathrm{IQC}(\rho,\Pi)\Rightarrow\Delta\in \mathrm{hardIQC}(\rho,\filt,\M)$.
    
    In the second part of the proof, we will show that~\eqref{eq:fdi} implies the existence of $\P\succ 0$ such that~\eqref{eq:lmi} holds.
    Due to~\cite[Corollary 12]{Boczar2015},~\eqref{eq:fdi} is equivalent to existence of $\P=\P^\top$ with $P\in \mathrm{LMI}(\rho, \filt, \M, \G)$. 
    This leads to $\rho^2\P\in \mathrm{LMI}(1, \hat \filt,\hat\M, \G_\rho)$ since $\hat \filt=\filt_\rho = \ssrep{\rho^{-1} \filtA}{\rho^{-1} \filtBin \ \ \rho^{-1} \filtBout}{\filtC}{\filtDin\ \ \filtDout}$, $\hat M=\M$, and $\G_\rho = \ssrep{\rho^{-1}\A_\K}{\rho^{-1}\Bw\ \ \rho^{-1}\Bd}{\C_\K}{\Dw\ \ \Dd}$.
    Since $\G_\rho\in\RHinf$ we can conclude with (ii) that $\rho^2\P\succeq 0$.
    Since the LMI holds strict, we can perturb $\P$ slightly to obtain $\P\succ 0$.
\end{proof}

\bibliographystyle{ieeetran}
\bibliography{bib_file}

\end{document}